\documentclass[twocolumn]{svjour3} 
\usepackage{graphicx}
\usepackage{latexsym}
\usepackage{amssymb}
\usepackage{subfigure}
\usepackage{fancyhdr}
\usepackage[usenames,dvipsnames]{pstricks}
\usepackage{epsfig}
\usepackage[english]{babel}
\usepackage[latin1]{inputenc}
\usepackage{indentfirst}
\usepackage{graphicx}
\usepackage{subfigure}
\usepackage{xspace} 
\usepackage{amsmath,amssymb}

\usepackage{fancyhdr}

\makeatletter
\newcommand{\sech}{\mathop{\operator@font sech}}
\newcommand{\sign}{\mathop{\operator@font sign}}
\makeatother

\numberwithin{equation}{section}

\newcommand{\divv}{{\rm{div}}}

\newcommand{\Var}{\mathrm{var}}

\DeclareMathOperator*{\esssup}{ess\,sup}
\DeclareMathOperator*{\essinf}{ess\,inf}
\journalname{Journal of Mathematical Imaging and Vision}
\begin{document}
\title{Cross-diffusion systems for image processing: II. The nonlinear case}
\author{A. Ara\'ujo         \and
        S. Barbeiro \and
E. Cuesta\and
A. Dur\'an
}
\authorrunning{Ara\'ujo {\em et al.}}
\institute{A. Ara\'ujo \at
              CMUC, Department of Mathematics, University of Coimbra, Portugal\\
              \email{alma@mat.uc.pt}           
           \and
           S. Barbeiro \at
              CMUC, Department of Mathematics, University of Coimbra, Portugal\\
\email{silvia@mat.uc.pt}
\and
E. Cuesta \at
              Department of Applied Mathematics,  University of
Valladolid,  Spain\\
\email{eduardo@mat.uva.es}
\and
A. Dur\'an \at
              Department of Applied Mathematics,  University of
Valladolid,  Spain\\
\email{angel@mac.uva.es}
}
\date{Received: date / Accepted: date}

\maketitle
\begin{abstract}
In this paper we study the application of $2\times 2$ nonlinear cross-diffusion systems as mathematical models of image filtering. These are systems of two nonlinear, coupled partial differential equations of parabolic type. The nonlinearity and cross-diffusion character are provided by a {nondiagonal matrix of diffusion coefficients that} depends on the variables of the system. We prove the well-posedness of an initial-boundary-value problem with Neumann boundary conditions and uniformly positive definite cross-diffusion matrix. Under additional hypotheses on the coefficients, the models are shown to satisfy the scale-space properties of shift, contrast, average grey and translational invariances. The existence of Lyapunov functions and the asymptotic behaviour of the solutions are also studied. According to the choice of the cross-diffusion matrix (on the basis of the results on filtering with linear cross-diffusion, discussed by the authors in a companion paper, and the use of edge stopping functions
) the performance of the models is compared by computational means in a filtering problem. The numerical results reveal differences {in the evolution of the filtering as well as in the quality of edge detection given by one of the components of the system, in terms of the cross-diffusion matrix}.
%
%
\end{abstract}

\section{Introduction}
\label{sec:sec1}
This paper is concerned with the use of nonlinear cross-diffusion systems for the mathematical modeling of image filtering. In this approach a grey-scale image is represented by a vector field ${\bf u}=(u,v)^{T}$ of two real-valued functions $u, v$ defined on some domain in $\mathbb{R}^{2}$. Additionally, an image restoration problem is modelled by an evolutionary process such that, from an initial distribution of a noisy image and with the time as a scale parameter, the restored image at any time satisfies an initial-boundary-value problem (IBVP) of a nonlinear system of partial differential equations (PDE) of cross-diffusion type, where the coupled evolution of the two components of the image and the nonlinearity are determined by a cross-diffusion coefficient matrix.

The use of cross-diffusion systems for modelling, especially in population dynamics, is well known, see e.~g. Galiano et al. \cite{GalianoGJ2001,GalianoGJ2003} and Ni \cite{Ni1998}  (along with references therein). To our knowledge, in the case of image processing, two previous proposals are related. The first one concerns the use of complex diffusion (Gilboa et al. \cite{GilboaSZ2004}), where the image is represented by a complex function and the filtering process is governed by a nonlinear PDE of diffusion type with a complex-valued diffusion coefficient. This equation can be written as a cross-diffusion system for the real and imaginary parts of the image. The application of complex diffusion to image filtering and edge-enhancing problems brings advantages based on the role of the imaginary part as edge detector in the linear case (the so-called small theta approximation) and its use, in the nonlinear case, instead of the size of the gradient of the image as the main variable to control the diffusion coefficient, Gilboa et al. \cite{GilboaZS2001,GilboaSZ2002,GilboaSZ2004}. 

A second reference on {nonlinear} cross-diffusion 
is the unpublished manuscript by Lorenz et al. \cite{LorenzBZ}, where the authors prove the existence of a global solution of a cross-diffusion problem, related to the complex diffusion approach proposed by Gilboa and collaborators. This already represents an advance with respect to the ill-posed Perona-Malik formulation, Perona \& Malik \cite{PeronaM1990} and Kinchenassamy \cite{Kichen1997}. Additionally, a better behaviour of cross-diffusion models with respect to the textures of the image is numerically suggested.

The present paper is a continuation of a companion work by the same authors devoted to the application of linear cross-diffusion processes to image filtering (Ara\'ujo et al. \cite{ABCD2016}). {The linear cross-diffusion is analyzed as a scale-space representation and an axiomatic, based on scale invariance, is built. Then those convolution kernels satisfying shift, rotational and scale invariance as well as recursivity (semigroup property) are characterized. The resulting filters are determined by a positive definite matrix, directing the diffusion, and a positive parameter which, as in the scalar case, Pauwels et al. \cite{PauwelsVFN1995}, delimits the locality property. Furthermore, since complex diffusion can be seen as a particular case of cross-diffusion, some properties of the former are generalized in the latter. More precisely, the use of one of the components of the cross-diffusion system as edge detector is investigated, extending the property of small theta approximation.}

The general purpose of the present paper is to continue the research on cross-diffusion models for image processing, by incorporating nonlinearity. The contributions of the paper are the following:
\begin{itemize}
\item We formalize nonlinear cross-diffusion IBVP as mathematical models for image processing, by proving the following theoretical results:
\begin{enumerate}
\item Well-posedness. By assuming that the coefficient matrix is uniformly positve definite and has globally Lipschitz and bounded entries, the IBVP of a nonlinear cross diffusion system of PDE with Neumann boundary conditions is  studied. The existence of  a unique weak solution, continuous dependence on the initial data and the existence of an extremum principle are proved. {Some of the arguments of Lorenz et al. \cite{LorenzBZ} for the system under study will be used and generalized here. Some extensions, not treated here, are the use of nonlocal operators and different types of boundary conditions in the PDE formulation.} 
\item The previous IBVP is also studied from the scale-space representation viewpoint, see e.~g. \'Alvarez et al. \cite{AlvarezGLM1993} and Lindeberg \cite{Linderberg2009}. Specifically, grey-level shift invariance, reverse contrast invariance and translational invariance are proved under additional assumptions on the diffusion coefficients.
\item The theoretical results are completed by analyzing the existence of Lyapunov functionals associated to the cross-diffusion problem, Weickert \cite{Weickert2}. The first result here is the decreasing of the energy (defined as the Euclidean norm of the solution) by cross-diffusion. The existence of Lyapunov functionals {different from this energy} depends on the relation between the cross-diffusion coefficient matrix and the function defining the functional. Finally, the solution is proved to evolve asymptotically to a constant image consisting of the average values of the components of the initial distribution.
\end{enumerate}
\item A numerical comparison of the performance of the models is made. The computational study is carried out on the basis of the results about the linear models, presented in Ara\'ujo et al. \cite{ABCD2016} and the numerical treatment of complex diffusion in Gilboa et al. \cite{GilboaSZ2004}. More precisely, the {performance of the experiments is based on the choices of the cross-diffusion coefficient matrix and the scheme of approximation to the continuous problem.}

As far as the coefficients are concerned, we select a matrix which combines linear cross-diffusion, including a constant positive definite matrix, with the use of standard edge detection functions, depending on the component of the image that plays the role of edge detector
from the generalized small theta approximation. The resulting form of the diffusion matrix  generalizes the complex diffusion approach, Gilboa et al. \cite{GilboaSZ2004}. Two strategies for the treatment of the edge detection functions are also implemented.

On the other hand, an adaptation to cross-diffusion systems of an explicit numerical method, considered and analyzed in Ara\'ujo et al. \cite{AraujoBS2012} and Bernardes et al. \cite{Bernardes:10}, for complex diffusion problems was used to perform the numerical experiments  in  filtering problems. The numerical results reveal differences in the behaviour of the models, according to the choice of the positive definite matrix and the edge stopping function. They are mainly concerned with a delay of the blurring effect (already observed in the linear case) and the influence of the generalized small theta approximation in the detection of the edges during the filtering problem.
\end{itemize}
The paper is structured according to these highlights. In Section \ref{sec:sec2} the IBVP of a cross-diffusion PDE with Neumann boundary conditions is introduced and the theoretical results of well-posedness, scale-space properties, Lyapunov functions and long time behaviour are proved. Section \ref{sec:sec3} is devoted to the computational study of the performance of the models. The main conclusions and future research are outlined in Section \ref{sec:sec4}.

The following notation will be used throughout the paper. A bounded (typically rectangular) domain in $\mathbb{R}^{2}$ will be denoted by $\Omega$, with boundary $\partial \Omega$ and where $\overline{\Omega}:=\Omega\cup \partial \Omega$. By ${\bf n}$ we denote the {outward} normal vector to $\partial \Omega$. For $p$ positive integer, $L^{p}(\Omega)$ denotes the normed space of $L^{p}$-functions on $\Omega$ with $||\cdot ||_{L^{p}}$ as the associated norm.
From the Sobolev space $H^{k}(\Omega)$ on $\Omega$ ($k$ is a nonnegative integer), where $H^{0}(\Omega)=L^{2}(\Omega)$ we define $X_{k}:=H^{k}(\Omega)\times H^{k}(\Omega)$ with norm denoted by
\begin{eqnarray*}
||{\bf u}||_{X_{k}}=\left(||u||_{{k}}^{2}+||v||_{{k}}^{2}\right)^{1/2},\quad {\bf u}=(u,v)^{T},
\end{eqnarray*}
where $||\cdot ||_{k}$ is the norm in $H^{k}(\Omega)$. On the other hand, the dual space of $H^{k}(\Omega)$ will be denoted by $\left(H^{k}(\Omega)\right)^{\prime}$; this is characterized as the completion of $L^{2}(\Omega)$ with respect to the norm, \cite{Adams1975},
\begin{eqnarray*}
||v||_{-k,2}=\sup_{u\in H^{k}(\Omega),||u||_{k}=1}|\langle u,v\rangle |,\quad
\langle u,v\rangle=\int_{\Omega}uvd\Omega.
\end{eqnarray*}
Additionally, $(X_{k})'$ will stand for $(H^{k}(\Omega))' \times (H^{k}(\Omega))'$.

For $T>0$, $Q_{T}=\Omega\times (0,T]$ will denote the set of points $({\bf x},t)$ with ${\bf x}\in\Omega, 0<t\leq T$ and $\overline{Q_{T}}:=\overline{\Omega}\times [0,T]$. The space of infinitely continuously differentiable {real-valued} functions in $\overline{\Omega}\times (0,T]$ will be denoted by $C^{\infty}\left(\overline{\Omega}\times (0,T]\right)$ as well as the space of $m-$th order continuously differentiable functions ${\bf u}:(0,T]\rightarrow X_{k}$ by $C^{m}(0,T,X_{k})$, $m,k$ nonnegative integers. Additionally, $L^{2}(0,T,H^{k})$  will stand for the normed space of functions $u:(0,T]\rightarrow H^{k}(\Omega)$ with associated norm
\begin{eqnarray*}
||u||_{L^{2}(0,T,H^{k})}=\left(\int_{0}^{T}||u(t)||_{k}^{2}dt\right)^{1/2}.
\end{eqnarray*}
We also denote by $L^{\infty}(0,T,H^{k})$ the normed space of functions $u:(0,T]\rightarrow H^{k}(\Omega)$ with norm
\begin{eqnarray*}
||u||_{L^{\infty}(0,T,H^{k})}=
\esssup_{t\in (0,T)}||u(t)||_{k},
\end{eqnarray*}
with $\esssup$ as the essential supremum. (The essential infimum will be denoted as $\essinf$.)

In Section \ref{sec:sec2} we will make use of the convolution operator
\begin{eqnarray}
(g\ast f)({\bf x})=\int_{\mathbb{R}^{2}}f({\bf x}-{\bf y})g({\bf y})d{\bf y},\label{conv}
\end{eqnarray}
for $g\in L^{1}(\mathbb{R}^{2}), f\in L^{2}(\mathbb{R}^{2})$ and the Fourier transform
\begin{eqnarray}
\widehat{f}({\bf \xi})=\int_{\mathbb{R}^{2}}f({\bf x})e^{-i{\bf \xi}\cdot{\bf x}}d{\bf x},\quad f\in L^{2}(\mathbb{R}^{2}), \quad \xi\in\mathbb{R}^{2}, \label{fourt}
\end{eqnarray}
where $\cdot$ denotes the Euclidean inner product in $\mathbb{R}^{2}$ with the norm represented by $| \cdot |$. In order to define (\ref{conv}), (\ref{fourt}) when $f\in L^{2}(\Omega)$, a continuous extension of $f$ in $\mathbb{R}^{2}$ will be {considered} and denoted by $\widetilde{f}$.

Finally, ${\divv}, \nabla$ will stand, respectively, for the divergence and gradient operators. Concerning the gradient, if ${\bf u}=(u,v)^{T}$ then $J{\bf u}$ stands for the Jacobian matrix of ${\bf u}$, ${\bf u}_{x}=(u_{x},v_{x})^{T}, {\bf u}_{y}=(u_{y},v_{y})^{T}$ and 
\begin{eqnarray*}
||J{\bf u}||_{X_{0}}:=\left(||{\bf u}_{x}||_{X_{0}}^{2}+||{\bf u}_{y}||_{X_{0}}^{2}\right)^{1/2}.
\end{eqnarray*}

Additional notation for the numerical experiments will be specified in Section \ref{sec:sec3}.

\section{Nonlinear cross-diffusion model}
\label{sec:sec2}
We consider the following IBVP of cross-diffusion for ${\bf u}=(u,v)^{T}$,
\begin{eqnarray}
\frac{\partial u}{\partial t}({\bf x},t)&=&{\divv}\left(D_{11}({\bf u}({\bf x},t))\nabla u({\bf x},t)\right.\nonumber\\
&&\left. +D_{12}({\bf u}({\bf x},t))\nabla v({\bf x},t)\right),\label{cd1}\\
\frac{\partial v}{\partial t}({\bf x},t)&=&{\divv}\left(D_{21}({\bf u}({\bf x},t))\nabla u({\bf x},t)\right.\nonumber\\
&&\left.+D_{22}({\bf u}({\bf x},t))\nabla v({\bf x},t)\right),\quad ({\bf x},t)\in Q_{T},\nonumber
\end{eqnarray}
with the initial data given by
\begin{eqnarray}
u({\bf x},0)=u_{0}({\bf x}), \quad v({\bf x},0)=v_{0}({\bf x}),\quad {\bf x}\in\Omega,\label{cd1a}
\end{eqnarray}
and Neumann boundary conditions in $\partial \Omega\times [0,T],$
\begin{eqnarray}
&& \langle D_{11}({\bf u})\nabla u+D_{12}({\bf u})\nabla v,n\rangle=0,\nonumber\\
&&\langle D_{21}({\bf u})\nabla u+D_{22}({\bf u})\nabla v,n\rangle=0.\label{cd1b}
\end{eqnarray}
In (\ref{cd1}), (\ref{cd1b}), 
the scalar functions $D_{ij}$, $i,j=1,2$, are the entries of a 
cross-diffusion {$2\times 2$} matrix operator 
$${\bf u}\mapsto D({\bf u}):\overline{Q_{T}}\rightarrow {M_{2\times 2}(\mathbb{R})},$$ with, for $ ({\bf x},t)\in \overline{Q_{T}},$
\begin{eqnarray*}
D({\bf u})({\bf x},t)=D({\bf u}({\bf x},t))=\begin{pmatrix}D_{11}({\bf u}({\bf x},t))&D_{12}({\bf u}({\bf x},t))\\
D_{21}({\bf u}({\bf x},t))&D_{22}({\bf u}({\bf x},t))\end{pmatrix},
\end{eqnarray*}
and which satisfies the following hypotheses:
\begin{itemize}
\item[(H1)] There exists $\alpha>0$ such that for each ${\bf u}:\overline{Q_{T}}\rightarrow\mathbb{R}^{2}$ 
\begin{eqnarray}
{\bf \xi}^{T}D({\bf u}({\bf x},t)){\bf \xi}\geq \alpha |\xi|^{2},\quad \xi\in\mathbb{R}^{2}, ({\bf x},t)\in \overline{Q_{T}}.\label{cd2}
\end{eqnarray}
\item[(H2)] There exists $L>0$ such that for ${\bf u},{\bf v}:\overline{Q_{T}}\rightarrow\mathbb{R}^{2}, ({\bf x},t)\in \overline{Q_{T}}, i,j=1,2,$
\begin{eqnarray*}
|D_{ij}({\bf v}({\bf x},t))-D_{ij}({\bf u}({\bf x},t))|\leq L |{\bf v}({\bf x},t)-{\bf u}({\bf x},t)|.
\end{eqnarray*}
\item[(H3)] There exists $M>0$ such that for each ${\bf u}:\overline{Q_{T}}\rightarrow\mathbb{R}^{2}$
\begin{eqnarray*}
|D_{ij}({\bf u}({\bf x},t))|\leq M,\quad ({\bf x},t)\in \overline{Q_{T}}, i,j=1,2.
\end{eqnarray*}
\end{itemize}
Conditions (H1)-(H3) will also be complemented with other assumptions, required by scale-space properties, see Section \ref{sec:sec22}.

In what follows the weak
formulation of (\ref{cd1})-(\ref{cd1b}) will be considered. This consists of finding ${\bf u}=(u,v)^{T}:(0,T]\longrightarrow X_{1}$ satisfying, for any $t\in (0,T]$
\begin{eqnarray}
&&\int_{\Omega} \left((\partial_{t}u)w_{1}+(\partial_{t}v)w_{2}\right)d\Omega\nonumber\\
&&+\int_{\Omega}\mathrm {tr}\left((J{\bf w})^{T}D({\bf u}) (J{\bf u})\right)d\Omega=0,\label{cd3}
\end{eqnarray}
for all ${\bf w}=(w_{1},w_{2})^{T}\in X_{1}$ and where $\mathrm {tr}$ denotes the trace of the matrix. 


\subsection{Well-posedness}
\label{sec:sec21}
This section is devoted to the study of well-posedness of (\ref{cd1})-(\ref{cd1b}). More precisely, we prove the existence of a unique solution of (\ref{cd3}), regularity, continuous dependence on the initial data and finally an extremum principle. The proofs follow standard arguments, see Catt\'e et al. \cite{CatteLMC1992}, Weickert \cite{Weickert2} (see also Galiano et al. \cite{GalianoGJ2001}  and references therein).
We first consider a related linear problem and prove a maximum-minimum principle as well as estimates of the solution in different norms. These results are crucial to prove the existence of the solution for the nonlinear case by using the Schauder fixed-point theorem, Brezis \cite{Brezis2011}. The same arguments as in Catt\'e et al. \cite{CatteLMC1992} and Weickert \cite{Weickert2}  apply to prove the uniqueness, as well as regularity and continuous dependence on the initial data. Finally, the proof of the extremum principle for the linear problem can be adapted to obtain the corresponding result for (\ref{cd1})-(\ref{cd1b}).

\begin{theorem}
\label{th1}
Let us assume that (H1)-(H3) hold and let ${\bf u}_{0}=(u_{0},v_{0})^{T}\in X_{1}$. Then (\ref{cd3}) admits a unique solution ${\bf u}\in C(0,T,X_{0})\cap L^{2}(0,T,X_{1})$ that depends continuously on the initial data. Furthermore, if $D$ is in $C^{\infty}(\mathbb{R}^{2},M_{2\times 2}(\mathbb{R}))$ then ${\bf u}$ is a strong solution of (\ref{cd1})-(\ref{cd1b}) with ${\bf u}\in C^{\infty}(\overline{\Omega}\times (0,T])$. 
\end{theorem}
\begin{proof}
We first define
\begin{eqnarray*}
W(0,T)&=&\{w\in L^{2}(0,T,H^{1}(\Omega)):\\
&&\frac{dw}{dt}\in L^{2}(0,T,(H^{1}(\Omega))') \} ,
\end{eqnarray*}
with the graph norm.

\subsubsection*{Existence}
\label{sec:sec211}
In order to study the existence of solution of (\ref{cd3}) we first consider, for 
${\bf U}=(U,V)^{T}$, with  
\begin{eqnarray*}
U,V\in W(0,T)\bigcap L^{\infty}(0,T,L^{2}(\Omega)),
\end{eqnarray*} the following linear IBVP in $Q_{T}$:
\begin{eqnarray}\label{cd7}
\frac{\partial u}{\partial t}({\bf x},t)&=&{\divv}\left(D_{11}({\bf U}({\bf x},t))\nabla u({\bf x},t)\right.\nonumber\\
&&\left.+D_{12}({\bf U}({\bf x},t))\nabla v({\bf x},t)\right),\nonumber\\
\frac{\partial v}{\partial t}({\bf x},t)&=&{\divv}\left(D_{21}({\bf U}({\bf x},t))\nabla u({\bf x},t)\right.\nonumber\\
&&\left.+D_{22}({\bf U}({\bf x},t))\nabla v({\bf x},t)\right),\nonumber\\
u({\bf x},0)&=&u_{0}({\bf x}), v({\bf x},0)=v_{0}({\bf x}),\quad {\bf x}\in\Omega,
\label{cd7a}
\end{eqnarray}
with Neumann boundary conditions in $\partial \Omega\times [0,T]$
\begin{eqnarray}
&& \langle D_{11}({\bf U})\nabla u+D_{12}({\bf U})\nabla v,n\rangle=0,\nonumber\\
&&\langle D_{21}({\bf U})\nabla u+D_{22}({\bf U})\nabla v,n\rangle=0.\label{cd7b}
\end{eqnarray}

Since $D({\bf U})=D(U,V)$ is uniformly positive definite (hypothesis (H1)), then, e.~g. Ladyzenskaya et al. \cite{LSU1968}, there is a unique weak solution of (\ref{cd7a}), (\ref{cd7b}), ${\bf u}(U,V)=(U_{1}(U,V),U_{2}(U,V))$, with $$U_{1},U_{2} \in W(0,T)\bigcap L^{\infty}(0,T,L^{2}(\Omega)).$$ We now establish some estimates for this solution in different norms, Lorenz et al.  \cite{LorenzBZ}. Consider first the weak formulation of (\ref{cd7}):
find ${\bf u}(U,V)=(U_{1}(U,V),U_{2}(U,V))$ in $L^{2}(0,T,X_{1})$ satisfying
\begin{eqnarray}
&&\int_{\Omega} \left((\partial_{t}U_{1})v_{1}+(\partial_{t}U_{2})v_{2}\right)d\Omega\nonumber\\
&&+\int_{\Omega}\mathrm{tr}\left((J{\bf v})^{T}D(U,V)(J{\bf u})\right)d\Omega=0,\label{ncd6}
\end{eqnarray}
for every ${\bf v}=(v_{1},v_{2})\in X_{1}$ and all $0\leq t\leq T$. We take the test functions $v_{1}=(U_{1}-b_1)_{+}, v_{2}=(U_{2}-b_{2})_{+}$ for some $b_1,b_2>0$ that will be specified later and where $f_{+}=\max \{f,0\}$ (Lorenz et al. \cite{LorenzBZ}, Weickert \cite{Weickert2}). Then (\ref{ncd6}) becomes
\begin{eqnarray*}
&&\frac{1}{2}\int_{\Omega} \left(\partial_{t}(U_{1}-b_1)_{+}^{2}+\partial_{t}(U_{2}-b_2)_{+}^{2}\right)d\Omega\\
&&+\int_{U_{1}>b_1,U_{2}>b_2}\mathrm{tr}\left((J{\bf u})^{T}D(U,V)(J{\bf u})\right)d\Omega=0.
\end{eqnarray*}
Then (H1) implies that
\begin{eqnarray*}
\frac{d}{dt}\int_{\Omega} \left((U_{1}-b_1)_{+}^{2}+(U_{2}-b_2)_{+}^{2}\right)d\Omega\leq 0.
\end{eqnarray*}
Thus integrating between $0$ and $t$, for any $0\leq t\leq T$, we have
\begin{eqnarray}
&&\int_{\Omega} \left((U_{1}(t)-b_1)_{+}^{2}+(U_{2}(t)-b_2)_{+}^{2}\right)d\Omega\nonumber\\
&&\leq
\int_{\Omega} \left((U_{1}(0)-b_1)_{+}^{2}+(U_{2}(0)-b_2)_{+}^{2}\right)d\Omega.\label{ncd7}
\end{eqnarray}
Now we take $b_1,b_2$ such that the integral on the right hand side of (\ref{ncd7}) becomes zero. If we assume that $U_{1}(0),U_{2}(0)\in L^{\infty}(\Omega)$ and define
$$ b_1=||U_{1}(0)||_{L^{\infty}(\Omega)}, \quad
b_2=||U_{2}(0)||_{L^{\infty}(\Omega)},
$$ then (\ref{ncd7}) implies
\begin{eqnarray*}
\int_{\Omega} \left((U_{1}(t)-b_1)_{+}^{2}+(U_{2}(t)-b_2)_{+}^{2}\right)d\Omega\leq
0
\end{eqnarray*}
and consequently
$(U_{1}(t)-b_1)_{+}=(U_{2}(t)-b_2)_{+}=0$ for $0\leq t\leq T$, that is
\begin{eqnarray}
&&U_{1}(x,t)\leq b_1=||U_{1}(0)||_{L^{\infty}(\Omega)},\nonumber\\
&&U_{2}(x,t)\leq b_2=||U_{2}(0)||_{L^{\infty}(\Omega)}.\label{ncd8}
\end{eqnarray}

Similarly, taking $v_{1}=(U_{1}-a_1)_{-}, v_{2}=(U_{2}-a_2)_{-}$ for some $a_1,a_2>0$ and where $f_{-}=\min \{f,0\}$, the same argument leads to
\begin{eqnarray*}
&&\int_{\Omega} \left((U_{1}(t)-a_1)_{-}^{2}+(U_{2}(t)-a_2)_{-}^{2}\right)d\Omega\\
&&\leq
\int_{\Omega} \left((U_{1}(0)-a_1)_{-}^{2}+(U_{2}(0)-a_2)_{-}^{2}\right)d\Omega.
\end{eqnarray*}
If we now {define}
$$ a_1=\essinf U_{1}(0), \quad
a_2=\essinf U_{2}(0),
$$
then
\begin{eqnarray*}
\int_{\Omega} \left((U_{1}(t)-a_1)_{-}^{2}+(U_{2}(t)-a_2)_{-}^{2}\right)d\Omega\leq
0
\end{eqnarray*}
and therefore
$(U_{1}(t)-a_1)_{-}=(U_{2}(t)-a_2)_{-}=0$ for $0\leq t\leq T$, that is
\begin{eqnarray}
U_{1}(x,t)\geq \essinf U_{1}(0),\; U_{2}(x,t)\geq \essinf U_{2}(0).\label{ncd9}
\end{eqnarray}
In particular, if $U_{1}(0),U_{2}(0)\geq 0$ then $U_{1}(x,t),U_{2}(x,t)\geq 0$ for all $(x,t)\in Q_{T}$.

\bigskip
A second estimate for the solution of the linear problem (\ref{cd7}) is now obtained from the functional of energy
$$
E_{L}(t)=\frac{1}{2}\int_{\Omega}\mathrm{tr}\left((J{\bf u})^{T}D(U,V)(J{\bf u})\right)d\Omega.
$$
Note that if in the weak formulation (\ref{ncd6}) we take ${\bf v}=(U_{1},U_{2})^{T}$ then
\begin{eqnarray*}
\frac{d}{dt}E_{L}(t)+\int_{\Omega}(\nabla U_{1} \nabla U_{2})D(U,V)(\nabla U_{1} \nabla U_{2})^{T}d\Omega=0,
\end{eqnarray*}
which implies
$$
\frac{d}{dt}E_{L}(t)\leq 0,
$$ that is $E_{L}(t)$ decreases. This leads to the $L^{\infty}$ estimates
\begin{eqnarray}
&&||U_{1}||_{L^{\infty}(0,T,L^{2}(\Omega))}\leq ||U_{1}(0)||_{L^{2}(\Omega)},\nonumber\\
&& ||U_{2}||_{L^{\infty}(0,T,L^{2}(\Omega))}\leq ||U_{2}(0)||_{L^{2}(\Omega)}.\label{ncd10}
\end{eqnarray}

We now search for estimates of $U_{1}(t), U_{2}(t)$ as functions in $H^{1}(\Omega)$ (and also of $\displaystyle\frac{d}{dt}U_{1}(t), \displaystyle\frac{d}{dt}U_{2}(t)$ as functions in $(H^{1}(\Omega))'$). Note first that from the previous argument we have, for $t\in [0,T]$,
\begin{eqnarray*}
&&\int_{\Omega}\left(U_{1}({\bf x},t)^{2}+U_{2}({\bf x},t)^{2}\right)d\Omega\\
&&\leq
\int_{\Omega}\left(U_{1}({\bf x},0)^{2}+U_{2}({\bf x},0)^{2}\right)d\Omega,
\end{eqnarray*}
and also
\begin{eqnarray}
&&\frac{d}{dt}\frac{1}{2}\int_{\Omega}\left(U_{1}({\bf x},t)^{2}+U_{2}({\bf x},t)^{2}\right)d\Omega\nonumber\\
&&+\alpha \int_{\Omega}\left(|\nabla U_{1}({\bf x},t)|^{2}+|\nabla U_{2}({\bf x},t)|^{2}\right)d\Omega \leq 0.\label{ncd11}
\end{eqnarray}
Then (\ref{ncd11}) implies that for any $t\in [0,T]$
\begin{eqnarray*}
&&\frac{1}{2}\int_{\Omega}\left(U_{1}({\bf x},t)^{2}+U_{2}({\bf x},t)^{2}\right)d\Omega\\
&&+\alpha \int_{0}^{t}  \int_{\Omega}\left(|\nabla U_{1}({\bf x},s)|^{2}+|\nabla U_{2}({\bf x},s)|^{2}\right)d\Omega ds\\
&&\leq \frac{1}{2}\int_{\Omega}\left(U_{1}({\bf x},0)^{2}+U_{2}({\bf x},0)^{2}\right)d\Omega.
\end{eqnarray*}
Therefore
\begin{eqnarray*}
&&\int_{0}^{T}\frac{1}{2}\int_{\Omega}\left(U_{1}({\bf x},t)^{2}+U_{2}({\bf x},t)^{2}\right)d\Omega dt\\
&&+\alpha \int_{0}^{T}  \int_{\Omega}\left(|\nabla U_{1}({\bf x},t)|^{2}+|\nabla U_{2}({\bf x},t)|^{2}\right)d\Omega dt\\
&&=\int_{0}^{T} E_{L}(t)dt\\
&&+\alpha \int_{0}^{T}  \int_{\Omega}\left(|\nabla U_{1}({\bf x},t)|^{2}+|\nabla U_{2}({\bf x},t)|^{2}\right)d\Omega dt\\
&&=\int_{0}^{T} E_{L}(t)dt-E_{L}(T)+E_{L}(T)\\
&&+\alpha \int_{0}^{T}  \int_{\Omega}\left(|\nabla U_{1}({\bf x},t)|^{2}+|\nabla U_{2}({\bf x},t)|^{2}\right)d\Omega dt\\
&&\leq \int_{0}^{T} E_{L}(t)dt-E_{L}(T)+E_{L}(0)\leq (T+1)E_{L}(0).
\end{eqnarray*}
Thus, if ${\bf U}_{0}=(U_{1}(0),U_{2}(0))^{T}$ then there exists a constant $C_{1}=C_{1}(\alpha,{\bf U}_{0},T)$ such that
\begin{eqnarray}
&&||U_{1}||_{L^{2}(0,T,H^{1}(\Omega))}\leq C_{1},\nonumber\\
&& ||U_{2}||_{L^{2}(0,T,H^{1}(\Omega))}\leq C_{1}.\label{ncd12}
\end{eqnarray}

On the other hand, if $||{\bf v}||_{L^{2}(0,T,X_{1})}=1$, the weak formulation (\ref{ncd6}), assumption (H3)  and Cauchy-Schwarz inequality imply that
\begin{eqnarray*}
&&\left|\int_{0}^{T}\int_{\Omega}\left((\partial_{t}U_{1})v_{1}+(\partial_{t}U_{2})v_{2}\right)d\Omega dt\right|\\
&&=\left| \int_{0}^{T}\left(\int_{\Omega}\mathrm{tr}\left((J{\bf v})^{T}D(U,V)(J{\bf u})\right)d\Omega\right)dt\right|\\
&&\leq \int_{0}^{T} M||\nabla {\bf v}(t)||_{X_{0}}||\nabla {\bf u}(t)||_{X_{0}}dt\\
&&\leq \int_{0}^{T} M||{\bf v}(t)||_{X_{1}}|| {\bf u}(t)||_{X_{1}}dt\\
&&\leq M||{\bf v}||_{L^{2}(0,T,X_{1})}|| {\bf u}||_{L^{2}(0,T,X_{1})}=M|| {\bf u}||_{L^{2}(0,T,X_{1})}.
\end{eqnarray*}
Therefore, this and (\ref{ncd12}) lead to
\begin{eqnarray}
||\frac{d}{dt} {\bf u}||_{L^{2}(0,T,(X_{1})')}\leq M || {\bf u}||_{L^{2}(0,T,X_{1})}\leq MC_{1}.\label{ncd13}
\end{eqnarray}
The existence of a solution of (\ref{cd3}) is now derived, making use of the estimates 
(\ref{ncd10}), (\ref{ncd12}) and (\ref{ncd13}) and by using the Schauder fixed-point theorem, Brezis \cite{Brezis2011}. (Analogous arguments were used in Catt\'e et al. \cite{CatteLMC1992}, see also Weickert \cite{Weickert2}.) We first assume that ${\bf u}_{0}=(u_{0},v_{0})^{T}\in X_{0}$ in (\ref{cd1a}). Consider the following subset of $W(0,T)^{2}:=W(0,T)\times W(0,T)$:
\begin{eqnarray*}
K&=&\{{\bf w}=(w_{1},w_{2})^{T}\in W(0,T)^{2}: {\bf w} \,\, \mbox{satisfies}\\
&&\,\, \mbox{(\ref{ncd10}), (\ref{ncd12}) and (\ref{ncd13})}\,\,\mbox{with}\,\, {\bf w}(0)={\bf u}_{0}\},
\end{eqnarray*}
and the mapping $T:K\longrightarrow W(0,T)^{2}$ such that $T({\bf w}):={\bf u}({\bf w})$ is the (weak) solution of (\ref{cd7}) with $(U,V)^{T}={\bf w}$.

It is not hard to see that $K$ is a nonempty, convex subset of $W(0,T)^{2}$. Our goal is to apply the Schauder fixed point theorem to the operator $T$ in the weak topology. To this end, we need to prove that:
\begin{itemize}
\item[(1)] $T(K)\subset K$.
\item[(2)] $K$ is a weakly compact subset of $W(0,T)^{2}$.
\item[(3)] $T$ is weakly continuous.
\end{itemize}
Observe that by construction (1) is satisfied. In order to prove (2), consider a sequence $\{{\bf w}_{n}\}_{n}\subset K$ and $t\in [0,T]$. Since $K$ is a bounded set, then $$\{{\bf w}_{n}(t)\}_{n}, \{\frac{d}{dt}{\bf w}_{n}(t)\}_{n}$$ are uniformly bounded in $X_{1}$ which implies the existence of a subsequence (denoted again by $\{{\bf w}_{n}(t)\}_{n},$  $\{\frac{d}{dt}{\bf w}_{n}(t)\}_{n}$) and ${\bf \varphi}(t),{\bf \psi}(t)\in X_{1}$ such that
$$
{\bf w}_{n}(t)\rightarrow {\bf \varphi}(t),\quad \frac{d}{dt}{\bf w}_{n}(t)\rightarrow {\bf \psi}(t),
$$ weakly in $X_{1}$ and for $0\leq t\leq T$. On the other hand, since $W(0,T)\subset L^{2}(0,T,L^{2}(\Omega))$ and the embedding is compact,  Catt\'e et al. \cite{CatteLMC1992}, there exists ${\bf w}\in L^{2}(0,T,X_{0})$ such that $||{\bf w}_{n}-{\bf w}||_{L^{2}(0,T,X_{0})}\rightarrow 0$ for some subsequence $\{{\bf w}_{n}\}_{n}$. Consequently ${\bf w}={\bf \varphi}\in L^{2}(0,T,X_{1})$. Actually, ${\bf \psi}=\frac{d}{dt}{\bf \varphi}$ and then $K$ is weakly compact in $W(0,T)^{2}$.

Finally, consider a sequence $\{{\bf w}_{n}\}_{n}\subset K$ which converges weakly to some ${\bf w}\in K$. Let ${\bf u}_{n}=T({\bf w}_{n})$. In order to prove property (3), we have to see that ${\bf u}_{n}$ converges weakly to ${\bf u}=T({\bf w})$. Here the proof is similar to that of Catt\'e et al.  \cite{CatteLMC1992}. Previous arguments applied to ${\bf u}_{n}$ and property (2) establish the existence of a subsequence  $\{{\bf u}_{n}\}_{n}$ and ${\bf \phi}\in L^{2}(0,T,X_{1})$ satisfying
\begin{itemize}
\item[(i)] ${\bf u}_{n}\rightarrow {\bf \phi}$ weakly in $L^{2}(0,T,X_{1})$;
\item[(ii)] $\frac{d}{dt}{\bf u}_{n}\rightarrow \frac{d}{dt}{\bf \phi}$ weakly in $L^{2}(0,T,(X_{1})')$;
\item[(iii)] ${\bf u}_{n}\rightarrow {\bf \phi}$ in $L^{2}(0,T,X_{0})$ and almost everywhere on $\Omega\times [0,T]$, (e.~g. Brezis \cite{Brezis2011}, Theorem 4.9);
\item[(iv)] ${\bf w}_{n}\rightarrow {\bf w}$  in $L^{2}(0,T,X_{0})$ and almost everywhere on $\Omega\times [0,T]$.
\end{itemize}
These convergence properties imply two additional ones:
\begin{itemize}
\item[(v)] ${\bf u}_{n}(0)\rightarrow {\bf \phi}(0)$ in $(X_{1})'$;
\item[(vi)] $\nabla {\bf u}_{n}\rightarrow \nabla {\bf \phi}$ weakly in $L^{2}(0,T,X_{0})$.
\end{itemize}
Now, note that due to (H2) and property (v) we have
$$D({\bf w}_{n})\rightarrow D({\bf w})$$ in $L^{2}(0,T,X_{0})$. Then if we take limit in (\ref{ncd6}) we have ${\bf \phi}=T({\bf w})$. Finally, since the whole sequence $\{{\bf u}_{n}\}_{n}$ is bounded in $K$ which is weakly compact, then it converges weakly in $W(0,T)$. By uniqueness of solution of (\ref{ncd6}) the whole sequence ${\bf u}_{n}=T({\bf w}_{n})$ must converge weakly to ${\bf \phi}=T({\bf w})$; therefore $T$ is weakly continuous and (3) holds.

Thus, Schauder fixed point theorem proves the existence of a solution ${\bf u}$ of (\ref{cd3}). The solution ${\bf u}$ is in $K$ and therefore ${\bf u}\in L^{2}(0,T,X_{1}), \frac{d{\bf u}}{dt}\in L^{2}(0,T,(X_{1})')$, and it satisfies (\ref{ncd10}), (\ref{ncd12}) and (\ref{ncd13}). Furthermore, due to the conditions (H1)-(H3) on $D$, at least ${\bf u}\in C(0,T,X_{0})$.

\subsubsection*{Regularity of solution}
\label{sec:sec212}
The same bootstrap argument as in Catt\'e et al. \cite{CatteLMC1992} and Weickert \cite{Weickert2} applies to obtain that ${\bf u}$ is a strong solution and ${\bf u}\in C^{\infty}(\overline{\Omega}\times (0,T])$ if (H2) is substituted by the hypothesis that $D$ is in $C^{\infty}(\mathbb{R}^{2},M_{2\times 2}(\mathbb{R}))$.
\subsubsection*{Uniqueness of solution}
\label{sec:sec213}
Consider ${\bf u}^{(1)}=(u^{(1)},v^{(1)})^{T}, {\bf u}^{(2)}=(u^{(2)},v^{(2)})^{T}$ solutions of (\ref{cd3}) with the same initial condition. Then for all ${\bf w}=(w_{1},w_{2})^{T}\in X_{1}$
\begin{eqnarray*}
&&\int_{\Omega} \left((\partial_{t}({u}^{(1)}-{u}^{(2)}))w_{1}+(\partial_{t}({ v}^{(1)}-{v}^{(2)}))w_{2}\right)d\Omega\\
&&+\int_{\Omega}\mathrm{tr}\left((J{\bf w})^{T}D({\bf u}^{(1)})(J{\bf u}^{(1)})\right)d\Omega\\
&&-\int_{\Omega}\mathrm{tr}\left((J{\bf w})^{T}D({\bf u}^{(2)})(J{\bf u}^{(2)})\right)d\Omega=0,
\end{eqnarray*}
which can be written as
\begin{eqnarray*}
&&\int_{\Omega} \left((\partial_{t}({u}^{(1)}-{u}^{(2)}))w_{1}+(\partial_{t}({ v}^{(1)}-{v}^{(2)}))w_{2}\right)d\Omega\\
&&+\int_{\Omega}\mathrm{tr}\left((J{\bf w})^{T}D({\bf u}^{(1)})(J({\bf u}^{(1)}-{\bf u}^{(2)}))\right)d\Omega\\
&&+\int_{\Omega}\mathrm{tr}\left((J{\bf w})^{T}\left(D({\bf u}^{(1)})-D({\bf u}^{(2)})\right)(J{\bf u}^{(2)})\right)d\Omega=0.
\end{eqnarray*}
Now we take ${\bf w}={\bf u}^{(1)}-{\bf u}^{(2)}$ and use (H1), (H2) to write
\begin{eqnarray*}
&&\frac{1}{2}\frac{d}{dt}||{\bf u}^{(1)}(t)-{\bf u}^{(2)}(t)||_{X_{0}}^{2}+\alpha
||J\left({\bf u}^{(1)}(t)-{\bf u}^{(2)}(t)\right)||_{X_{0}}^{2}\\
&&\leq L||{\bf u}^{(1)}(t)-{\bf u}^{(2)}(t)||_{X_{0}}||J {\bf u}^{(2)}(t)||_{X_{0}}\\
&&||J \left({\bf u}^{(1)}(t)-{\bf u}^{(2)}(t)\right)||_{X_{0}}\\
&&\leq \frac{1}{\alpha}L^{2}||{\bf u}^{(1)}(t)-{\bf u}^{(2)}(t)||_{X_{0}}^{2}||J {\bf u}^{(2)}(t)||_{X_{0}}^{2}\\
&&+\frac{\alpha}{4}||J\left({\bf u}^{(1)}(t)-{\bf u}^{(2)}(t)\right)||_{X_{0}}^{2}.
\end{eqnarray*}
(In the last step the inequality $ab\leq a^{2}/4\epsilon^{2}+\epsilon^{2}b^{2}$ has been used, with $\epsilon^{2}=\alpha/4$.) Therefore
\begin{eqnarray*}
&&\frac{d}{dt}||{\bf u}^{(1)}(t)-{\bf u}^{(2)}(t)||_{X_{0}}^{2}\\
&&\leq
 \frac{2}{\alpha}L^{2}||{\bf u}^{(1)}(t)-{\bf u}^{(2)}(t)||_{X_{0}}^{2}||J {\bf u}^{(2)}(t)||_{X_{0}}^{2}.\label{ncd14}
\end{eqnarray*}
Finally, Gronwall's lemma leads to
\begin{eqnarray}
&&||{\bf u}^{(1)}(t)-{\bf u}^{(2)}(t)||_{X_{0}}^{2}\label{ncd15}\\
&&\leq
||{\bf u}^{(1)}(0)-{\bf u}^{(2)}(0)||_{X_{0}}^{2}\mathrm {exp}\left(C\int_{0}^{t}
||J{\bf u}^{(2)}(s)||_{X_{0}}^{2}ds\right),\nonumber
\end{eqnarray}
with $C=\frac{2}{\alpha}L^{2}$
and since ${\bf u}^{(1)}(0)={\bf u}^{(2)}(0)$ then uniqueness is proved.
%
%

\subsubsection*{Continuous dependence on initial data}
\label{sec:sec215}
Since ${\bf u}$ is bounded on $\overline{Q_{T}}$, then $J{\bf u}$ is bounded and hypothesis (H1) on $D$ implies
\begin{eqnarray*}
&&\int_{0}^{t}||J {\bf u}(\cdot,s)||_{X_{0}}^{2}ds\\
&&\leq 
\int_{0}^{T}||J {\bf u}(\cdot,s)||_{X_{0}}^{2}ds
=\frac{1}{\alpha}\int_{0}^{T}\alpha ||J {\bf u}(\cdot,s)||_{X_{0}}^{2}ds\\
&\leq &\frac{1}{\alpha}\left|\int_{0}^{T}\int_{\Omega}\nabla {\bf u}({\bf x},t) D({\bf u}({\bf x},t))\nabla{\bf u}({\bf x},t)^{T}d\Omega\right|ds\\
&=&\frac{1}{\alpha}\left|\int_{0}^{T}\int_{\Omega}{\bf u}({\bf x},t)^{T} {\bf u}_{t}({\bf x},t)d\Omega\right|ds\\
&&\leq \frac{1}{\alpha}\int_{0}^{T}||{\bf u}(\cdot,s)||_{X_{0}} ||{\bf u}_{t}(\cdot,s)||_{X_{0}}ds\\
&\leq & \frac{1}{\alpha}||{\bf u}||_{L^{2}(0,T,X_{1})} ||{\bf u}_{t}||_{L^{2}(0,T,(X_{1})')}.
\end{eqnarray*}
Now, let $\epsilon>0$ and take
$$
\delta:=\epsilon \,\mathrm {exp}\left(-\frac{C}{\alpha}||u(s)||_{L^{2}(0,T,X_{1})} ||u_{t}||_{L^{2}(0,T,(X_{1})')}\right).
$$ If $||{\bf u}^{(1)}(0)-{\bf u}^{(2)}(0)||_{X_{0}}<\delta$ and using (\ref{ncd15}) then
$$||{\bf u}^{(1)}(t)-{\bf u}^{(2)}(t)||_{X_{0}}<\epsilon,$$ for all $t\in [0,T]$. This proves the continuous dependence on the initial data.
\end{proof}
\subsubsection*{Extremum principle}
\label{sec:sec214}
Well-posedness results are finished off with the following extremum principle.
\begin{theorem}
\label{th2}
Let us assume that in (\ref{cd1a}) ${\bf u}_{0}=(u_{0},v_{0})^{T}\in L^{\infty}(\Omega)\times L^{\infty}(\Omega)$ and define:
\begin{eqnarray*}
&&a_{1}=\essinf u_{0}, \quad
a_{2}=\essinf v_{0},
\\ &&b_{1}=||u_{0}||_{L^{\infty}(\Omega)}, \quad
b_{2}=||v_{0}||_{L^{\infty}(\Omega)}.
\end{eqnarray*}
Let ${\bf u}=(u,v)^{T}$ be the weak solution of (\ref{cd1})-(\ref{cd1b}). Then for all $({\bf x},t)\in Q_{T}$ 
\begin{eqnarray*}
a_{1}\leq u({\bf x},t)\leq b_{1},\quad
a_{2}\leq v({\bf x},t)\leq b_{2}.
\end{eqnarray*}

\end{theorem}

\begin{proof}. Note that the same argument as that of the linear problem (\ref{cd7}) can be adapted to this nonlinear case straightforwardly, by taking, in the case of the maximum principle, $w_{1}=(u-b_{1})_{+}, w_{2}=(v-b_{2})_{+}$ in the weak formulation (\ref{cd3}) and, in the case of the minimum principle, $w_{1}=(u-a_{1})_{-}, w_{2}=(v-a_{2})_{-}$.
\end{proof}

\begin{remark}
\label{rem1}
In Gilboa et al. \cite{GilboaSZ2004}, a nonlinear complex diffusion problem with diffusion coefficient of the form
\begin{eqnarray}
c=c(v)=\frac{e^{i\theta}}{1+\left(\frac{v}{\kappa \theta}\right)^{2}},\label{cdr1}
 \end{eqnarray}
is considered. In (\ref{cdr1}) the image is represented by a complex function $u+iv$, $\kappa$ is a threshold parameter and $\theta$ is a phase angle parameter. In the corresponding cross-diffusion formulation (\ref{cd1}) for ${\bf u}=(u,v)^{T}$, the coefficient matrix is
\begin{eqnarray}
D(u,v)&=&g(v)\begin{pmatrix}\cos\theta &-\sin\theta\\ \sin\theta&\cos\theta\end{pmatrix},\nonumber\\ 
g(v)&=&\frac{1}{1+\left(\frac{v}{\kappa \theta}\right)^{2}}.\label{cdr2}
\end{eqnarray}
Thus, for $\xi\in\mathbb{R}^{2}$,
\begin{eqnarray*}
\xi^{T}D(u,v)\xi=(g(v)\cos\theta)|\xi|^{2}.
\end{eqnarray*}
The function $g$ in (\ref{cdr2}) is decreasing for $v\geq 0$ and satisfies $g(0)=1$, $\lim_{s\rightarrow +\infty} g(s)=0$. Consequently, $D$ in (\ref{cdr2}) would not satisfy (H1) for $v\geq 0$. In addition to assuming $\theta\in (0,\pi)$ (in order to have $\cos\theta>0$), two strategies to overcome this drawback are suggested. 
\begin{itemize}
\item
The first one is to  replace $g(v)$ in (\ref{cdr2}) by $g(M(v))$, where $M(\cdot)$ is a cut-off operator
\begin{equation}\label{cdr3}
M(v)({\mathbf x},t)=\min_{({\mathbf x},t)\in \overline{Q_T}}\{v({\mathbf x},t),M\},
\end{equation}
with $M$ a sufficiently large constant. The same approach can be generalized for the cross-diffusion matrix operator $D$.
\item A second {strategy} is to replace $g(v)$ in (\ref{cdr2}) by $g(|w_{\sigma}|)$ where $w_{\sigma}$ is the second component of the matrix convolution ${\bf v}_{\sigma}=K_{\sigma}\ast {\bf u}$, $K_{\sigma}:=K(\cdot,\sigma)=(k_{ij}(\cdot,\sigma))_{i,j=1,2}$ is the matrix such that (Ara\'ujo et al. \cite{ABCD2016})
\begin{eqnarray*}
\widehat{K}_{\sigma}(\xi)=(\widehat{k}_{ij}(\cdot,\sigma))_{i,j=1,2}=e^{-|\xi |^{2}\sigma d},\quad \xi\in\mathbb{R}^{2},\label{cdr4}
\end{eqnarray*}
where
\begin{eqnarray*}
d=d_{\theta}=\begin{pmatrix}\cos\theta &-\sin\theta\\ \sin\theta&\cos\theta\end{pmatrix}.
\end{eqnarray*}
(The matrix convolution $K_{\sigma}\ast {\bf u}$ is defined as the vector
\begin{eqnarray*}
\begin{pmatrix}k_{11}\ast \widetilde{u}+k_{12}\ast \widetilde{v}\\
k_{21}\ast \widetilde{u}+k_{22}\ast \widetilde{v}
\end{pmatrix}
\end{eqnarray*}
where $\ast$ denotes the usual convolution operator in $\mathbb{R}^{2}$ and ${\bf \widetilde{u}}=(\widetilde{u},\widetilde{v})^{T}$ is a continuous extension of ${\bf u}$ in $\mathbb{R}^{2}$.)
 
We observe that the weak formulation (\ref{cd3}) with the corresponding modified matrix $D(u,v)=g(|w_{\sigma}|)d_{\theta}$ satisfies the conclusions of Theorem \ref{th1} by adapting the proof as follows (see Catt\'e et al. \cite{CatteLMC1992}): Let $U,V\in W(0,T)\bigcap L^{\infty}(0,T,H^{0}(\Omega))$ such that
\begin{eqnarray*}
&&||U||_{L^{\infty}(0,T,H^{0}(\Omega))}\leq ||u_{0}||_{0},\\
&&||V||_{L^{\infty}(0,T,H^{0}(\Omega))}\leq ||v_{0}||_{0}.
\end{eqnarray*}
Since $U,V\in L^{\infty}(0,T,H^{0}(\Omega))$ and $g$ as well as each entry of $K_{\sigma}$ are $C^{\infty}$ then $g(|w_{\sigma}|)\in L^{\infty}(0,T,C^{\infty}(\Omega))$. Thus, since $g$ is decreasing, there is $C>0$, which only depends on $g, K_{\sigma}$ and $||u_{0}||_{0}, ||v_{0}||_{0}$ such that
\begin{eqnarray*}
g(|w_{\sigma}|)\leq C
\end{eqnarray*}
almost everywhere in $Q_{T}$. Thus the corresponding matrix $D(u,v)=g(|w_{\sigma}|)d_{\theta}$ satisfies (H1) for almost any $({\bf x},t)\in Q_{T}$. With this modification, the rest of Theorem \ref{th1} is proved in the same way.
\end{itemize}
The previous argument can be generalized to a general cross-diffusion problem (\ref{cd1}) with cross-diffusion matrices of the form
\begin{eqnarray}
D({\bf u})=D(u,v)=g(|w_{\sigma}|)d,\label{cdr5}
\end{eqnarray}
where 
\begin{itemize}
\item[(i)]
$w_{\sigma}$ is the second component of ${\bf v}_{\sigma}=K_{\sigma}\ast {\bf u}$ with $K_{\sigma}$ satisfying
\begin{eqnarray*}
\widehat{K}_{\sigma}(\xi)=e^{-|\xi |^{2}\sigma d},\quad \xi\in\mathbb{R}^{2},\label{cdr6}
\end{eqnarray*}
for some positive definite matrix $d$ and,
\item[(ii)] $g:[0,+\infty) \longrightarrow (0,+\infty)$ is a smooth, decreasing function with $g(0)=1$ and $\lim_{s\rightarrow +\infty} g(s)=0$.
\end{itemize}
\end{remark}

\subsection{Scale-space properties}
\label{sec:sec22}
For $t\geq 0$ let us define the scale-space operator
\begin{eqnarray}
T_{t}:{\bf u}_{0}\longmapsto T_{t}({\bf u}_{0}):={\bf u}(t)={\bf u}(\cdot,t),\label{cd221}
\end{eqnarray}
such that
${\bf u}(t)$ is the unique weak solution at time $t$ of (\ref{cd1})-(\ref{cd1b}) with initial data (\ref{cd1a}) given by ${\bf u}_{0}$. Some properties of (\ref{cd221}) will be here analyzed. More precisely additional hypotheses on $D$ in (\ref{cd1}) allow (\ref{cd221}) to satisfy grey-level shift, reverse constrast, average grey and translational invariances. {In what follows we assume that (H1)-(H3) hold.}
\subsubsection{Grey-level shift invariance}
\label{sec:sec221}
\begin{lemma}
\label{lem1} Let us assume that $D$ in (\ref{cd1}) additionally satisfies
\begin{eqnarray}
D({\bf u}({\bf x},t)+{\bf C})=D({\bf u}({\bf x},t)),\label{cd222}
\end{eqnarray}
for all $({\bf x},t)\in Q_{T}, {\bf u}(\cdot, t)\in X_{1}$ and ${\bf C}=(C_{1},C_{2})^{T}\in \mathbb{R}^{2}$. Then 
\begin{eqnarray}
T_{t}({\bf 0})={\bf 0},\quad T_{t}({\bf u}_{0}+{\bf C})=T_{t}({\bf u}_{0})+{\bf C},\quad t\geq 0.\label{ncd16}
\end{eqnarray}
\end{lemma}

\begin{proof}
The main argument for the proof is the uniqueness of solution of (\ref{cd1})-(\ref{cd1b}). Note first that ${\bf u}={\bf 0}$ is a solution with ${\bf u}_{0}={\bf 0}$ and consequently 
it is clear that $T_{t}({\bf 0})={\bf 0}$. On the other hand, because of (\ref{cd222}) we have that
\begin{eqnarray*}
{\bf w}(t)=T_{t}({\bf u}_{0})+{\bf C},\quad t\geq 0,
\end{eqnarray*} satisfies (\ref{cd1}) with initial condition ${\bf u}_{0}+{\bf C}$ and therefore, by uniqueness, it must coincide with $T_{t}({\bf u}_{0}+C)$.
\end{proof}

\begin{remark}
From Ara\'ujo et al. \cite{ABCD2016}, we know that the kernel matrices $K_{\sigma}$ satisfying (\ref{cdr6}) are mass preserving, that is $K_{\sigma}\ast {\bf C}={\bf C}, {\bf C}\in \mathbb{R}^{2}$ and therefore
\begin{eqnarray*}
K_{\sigma}\ast ({\bf u}+{\bf C})=(K_{\sigma}\ast {\bf u})+{\bf C}.
\end{eqnarray*}
This implies that cross-diffusion coefficient matrices  (\ref{cdr5}) satisfy (\ref{ncd16}) but only for constants ${\bf C}=(C_{1},0)^{T}, C_{1}\in\mathbb{R}$. If the first component of ${\bf u}(t)$ represents the grey-level values of the filtered image at time $t$ then this weaker version of (\ref{ncd16}) can be interpreted as shift invariance of the grey values.
\end{remark}

\subsubsection{Reverse contrast invariance}
\label{sec:sec222}
\begin{lemma}
\label{lem2} Let us assume that $D$ in (\ref{cd1}) additionally satisfies
\begin{eqnarray}
D(-{\bf u}({\bf x},t))=D({\bf u}({\bf x},t)),\label{cd223}
\end{eqnarray}
for all $({\bf x},t)\in Q_{T}, {\bf u}(\cdot, t)\in X_{1}$. Then 
\begin{eqnarray}
T_{t}(-{\bf u}_{0})=-T_{t}({\bf u}_{0}),\quad t\geq 0.\label{ncd16b}
\end{eqnarray}
\end{lemma}
\begin{proof}
By (\ref{cd223}) the functions ${\bf w}_{1}(t)=T_{t}(-{\bf u}_{0})$ and ${\bf w}_{2}(t)=-T_{t}(-{\bf u}_{0})$ satisfy (\ref{cd3}) with the same initial data $-{\bf u}_{0}$. Therefore, by uniqueness, ${\bf w}_{1}={\bf w}_{2}$ and (\ref{ncd16b}) holds.
\end{proof}

\begin{remark}
Matrices $D$ of the form (\ref{cdr5}) satisfy (\ref{cd223}) and therefore the corresponding operator (\ref{cd221}) satisfies (\ref{ncd16b}).
\end{remark}
\subsubsection{Average grey invariance}
\label{sec:sec223}
For $f\in L^{2}(\Omega)$ we define
\begin{eqnarray*}
m(f)=\frac{1}{A(\Omega)}\int_{\Omega} f({\bf x})d\Omega,
\end{eqnarray*}
where $A(\Omega)$ stands for the area of $\Omega$.
\begin{lemma}
\label{lem3} For ${\bf u}=(u,v)^{T}$ let ${\bf M}({\bf u})=(m(u),m(v))^{T}$. Then
\begin{eqnarray}
{\bf M}(T_{t}({\bf u}_{0}))={\bf M}({\bf u}_{0}),\quad t\geq 0.\label{ncd18}
\end{eqnarray}
\end{lemma}
\begin{proof}
We consider the vector function
\begin{eqnarray*}
&&{\bf G}(t)=(G_{1}(t),G_{2}(t))^{T}, \\
&& G_{1}(t)=\int_{\Omega} u({\bf x},t)d\Omega,
\quad G_{2}(t)=\int_{\Omega} v({\bf x},t)d\Omega,\; t\geq 0,
\end{eqnarray*}
where ${\bf u}=(u,v)^{T}=T_{t}({\bf u}_{0})$.
As in Weickert \cite{Weickert2}, we have, for $i=1,2$,
\begin{eqnarray*}
|G_{i}(t)-G_{i}(0)|\leq A(\Omega)^{1/2}||{\bf u}(t)-{\bf u}(0)||_{X_{0}}.
\end{eqnarray*}
Since at least ${\bf u}\in C(0,T,X_0)$ then ${\bf G}$ is continuous at $t=0$. On the other hand, divergence theorem and the boundary conditions (\ref{cd1b}) imply that, for $ i=1,2$
\begin{eqnarray*}
\frac{d}{dt}G_{i}(t)
&=&\int_{\Omega} {\divv}(D_{i1}({\bf u})\nabla u+D_{i2}({\bf u})\nabla v)d\Omega\\
&=&\int_{\partial \Omega} \langle D_{i1}({\bf u})\nabla u+D_{i2}({\bf u})\nabla v), n \rangle d\Gamma=0.
\end{eqnarray*}
Then $G_{i}(t)$ is constant for all $t\geq 0$. Thus the quantity
$
{\bf M}({\bf u}_{0})=(m(u_{0}),m(v_{0}))^{T}$
is preserved by cross-diffusion.
\end{proof}
\begin{remark}
Actually, each component $G_{i}(t), i=1,2$ is preserved. This may be used to establish a suitable definition of average grey level in this formulation, using these two quantities, and its preservation by cross-diffusion; we refer Ara\'ujo et al. \cite{ABCD2016} for a discussion about this question.
\end{remark}
\subsubsection{Translational invariance}
\label{sec:sec224}
Let us define the translational operator $\tau_{{\bf h}}$ as
\begin{eqnarray*}
\tau_{{\bf h}}{\bf u}({\bf x})&=&(u({\bf x}+{\bf h}),v({\bf x}+{\bf h}))^{T},\\
&& {\bf u}=(u,v)^{T}\in X_{0}, \quad {\bf x},{\bf h}\in \mathbb{R}^{2}.
\end{eqnarray*} 
\begin{lemma}
\label{lem4}
We assume that $D$ in (\ref{cd1}) additionally satisfies
\begin{eqnarray}
D(\tau_{{\bf h}}{\bf u}({\bf x},t))=D({\bf u}({\bf x},t)),\label{cd224}
\end{eqnarray}
for all $({\bf x},t)\in Q_{T}, {\bf u}(\cdot, t)\in X_{1}$. Then 
\begin{eqnarray}
T_{t}(\tau_{{\bf h}}{\bf u}_{0})=\tau_{{\bf h}}(T_{t}({\bf u}_{0})),\quad t\geq 0.\label{ncd19}
\end{eqnarray}
\end{lemma}
\begin{proof}
Due to (\ref{cd224}), the functions ${\bf w}_{1}(t)=T_{t}(\tau_{{\bf h}}{\bf u}_{0})$ and ${\bf w}_{2}(t)=\tau_{{\bf h}}T_{t}({\bf u}_{0})$ are solutions of (\ref{cd3}) with the same initial data $\tau_{{\bf h}}{\bf u}_{0}$. Therefore, by uniqueness, (\ref{ncd19}) holds.
\end{proof}

\begin{remark}
Matrices $D$ of the form (\ref{cdr5}) satisfy (\ref{cd224}) and therefore the corresponding operator (\ref{cd221}) satisfies (\ref{ncd19}).
\end{remark}
\subsection{Lyapunov functions and behaviour at infinity}
\label{sec:sec23}
The previous study is finished off by analyzing the existence of Lyapunov functionals and the behaviour of the solution when $t\rightarrow +\infty$. As far as the first question is concerned, we have the following result.
\begin{lemma}
\label{lem5} Let ${\bf u}=(u,v)^{T}$ be the unique weak solution of (\ref{cd1})-(\ref{cd1b}) and let us consider the functional
\begin{eqnarray}
V(t)=\Phi({\bf u}(t)):=\frac{1}{2}\int_{\Omega}\left(u({\bf x},t)^{2}+v({\bf x},t)^{2}\right)d\Omega.\label{cd9}
\end{eqnarray}
Then $V$ defines a Lyapunov function for (\ref{cd1})-(\ref{cd1b}).
\end{lemma}
\begin{proof}
Note first that from the weak formulation (\ref{cd3}) with ${\bf w}={\bf u}$ we obtain
\begin{eqnarray*}
\frac{d}{dt}V(t)+
\int_{\Omega}{\rm tr}\left((J{\bf u})^{T}D({\bf u}({\bf x},t))(J{\bf u})\right)d\Omega=0,
\end{eqnarray*}
which, due to (H1) and  (\ref{cd3}), implies
\begin{eqnarray*}
\frac{d}{dt}V(t)\leq 0,\quad t\geq 0.
\end{eqnarray*}
Note also that since $\widetilde{r}(z)={z^{2}}/{2}$ is convex and (\ref{ncd18}) holds, then Jensen inequality implies that
\begin{eqnarray*}
\Phi(M{\bf u}_{0})&=&\int_{\Omega}\frac{(m(u_{0}))^{2}+(m(v_{0}))^{2}}{2}d\Omega\\
&=&\int_{\Omega}\widetilde{r}(m(u_{0}))+\widetilde{r}(m(v_{0}))d\Omega\\
&=&\int_{\Omega}\widetilde{r}(m(u(t)))+\widetilde{r}(m(v(t)))d\Omega\\
&\leq &\int_{\Omega}\left(\frac{1}{A(\Omega)}\int_{\Omega} \widetilde{r}(u({\bf x},t))d\Omega\right.\\
&&\left.+\frac{1}{A(\Omega)}\int_{\Omega} \widetilde{r}(v({\bf x},t))d\Omega\right)d\Omega\\
&=&\int_{\Omega}\left( \widetilde{r}(u({\bf x},t)+ \widetilde{r}(v({\bf x},t)\right)d\Omega=\Phi({\bf u}(t)).
\end{eqnarray*}
Therefore, (\ref{cd9}) is a Lyapunov functional. 
\end{proof}

{The search for more Lyapunov functions will make use of convex functions.
\begin{lemma}
\label{lem6}
Let $a_{j}, b_{j}, j=1,2$ be defined in Theorem \ref{th2}, $I=I({\bf a},{\bf b})=(a_{1},b_{1})\times (a_{2},b_{2})$ and let us assume that $r$ is a $C^{2}(I)$ convex function satisfying
\begin{eqnarray}
\nabla^{2}r({\bf u}({\bf x}))D({\bf u}({\bf x}))=D({\bf u}({\bf x}))\nabla^{2}r({\bf u}({\bf x})),\label{cd226}
\end{eqnarray}
where ${\bf u}$ is the weak solution of (\ref{cd1})-(\ref{cd1b}) with $u_{0},v_{0}\geq 0$ and $\nabla^{2}r(u,v)$ stands for the Hessian of $r$. Then 
\begin{eqnarray}
V_{r}(t)=\Phi_{r}({\bf u}(t))=\int_{\Omega} r(u(t),v(t))d\Omega,\label{cd225}
\end{eqnarray}
is a Lyapunov function for (\ref{cd1})-(\ref{cd1b}).
\end{lemma}
\begin{proof}
Observe that using divergence theorem, boundary conditions (\ref{cd1b}) and after some computations we have 
\begin{eqnarray*}
V_{r}^{\prime}(t)&=&\int_{\Omega}\left(r_{u}(u,v)u_{t}+r_{v}(u,v)v_{t}\right)d\Omega\\
&=&\int_{\Omega}\left(r_{u}(u,v){\divv}\left(D_{11}({\bf u})\nabla u+D_{12}({\bf u})\nabla v\right)\right.\\
&&\left. +r_{v}(u,v){\divv}\left(D_{21}({\bf u})\nabla u+D_{22}({\bf u})\nabla v\right)\right)d\Omega\\
&=&-\int_{\Omega}\left(\langle \nabla^{2}r(u,v)\begin{pmatrix}u_{x}\\v_{x}\end{pmatrix},D({\bf u})\begin{pmatrix}u_{x}\\v_{x}\end{pmatrix}\rangle\right.\\
&&+\left.
\langle \nabla^{2}r(u,v)\begin{pmatrix}u_{y}\\v_{y}\end{pmatrix},D({\bf u})\begin{pmatrix}u_{y}\\v_{y}\end{pmatrix}\rangle\right)d\Omega.
\end{eqnarray*}
Since $r$ is convex, then $\nabla^{2}r(u,v)$ is positive semi-definite. Thus, due to (\ref{cd226}), $\nabla^{2}r(u,v)D({\bf u})$ is positive semi-definite and therefore $V^{\prime}(t)\leq 0, t\geq 0$.
Similarly, the application of a generalized version of Jensen inequality, Zabandan \& Kiliman \cite{Zabandan2012}, and the convexity of $r$ imply
$$r({\bf M}({\bf u}))\leq m(r({\bf u})).$$ This and (\ref{ncd18}) lead to
\begin{eqnarray*}
\Phi_{r}({\bf M}({\bf u}_{0}))&=&\int_{\Omega}r({\bf M}({\bf u}_{0}))d\Omega=\int_{\Omega}r({\bf M}({\bf u}(t)))d\Omega\\
&\leq & \int_{\Omega}m(r({\bf u}(t)))d\Omega\\
&=& \int_{\Omega}\frac{1}{A(\Omega)}\int_{\Omega}r({\bf u}(t))d{\bf x}d\Omega
=\int_{\Omega}r({\bf u}(t))d{\bf x}\\
&=&\Phi_{r}({\bf u}(t)),
\end{eqnarray*}
and (\ref{cd225}) is a Lyapunov functional.
\end{proof}
\begin{remark}
The choice $r(x,y)=\frac{x^{2}+y^{2}}{2}$ leads to the Lyapunov functional (\ref{cd9}). More generally, if $p\geq 2$ then taking
$$r(x,y)=|x|^{p}+|y|^{p},$$ implies that the $L^{p}\times L^{p}$ norm
\begin{eqnarray*}
||{\bf u}||_{L^{p}\times L^{p}}=\left(||u||_{L^{p}}^{p}+||v||_{L^{p}}^{p}\right),\quad {\bf u}=(u,v)^{T},
\end{eqnarray*}
is a Lyapunov functional, see Weickert \cite{Weickert2}.
\end{remark}
}

As far as the behaviour at infinity of the solution of (\ref{cd1})-(\ref{cd1b}) is concerned, the arguments in Weickert \cite{Weickert2} can also be adapted here. 
\begin{lemma}
\label{lem7}
Let ${\bf u}(t), t\geq 0$ be the weak solution of (\ref{cd1})-(\ref{cd1b}) and let us consider ${\bf w}={\bf u}-{\bf M}({\bf u}_{0})$, where ${\bf M}$ is given by Lemma \ref{lem3}. If (\ref{ncd16}) holds then
\begin{eqnarray}
\lim_{t\rightarrow\infty}||{\bf w}(t)||_{X_{0}}=0.\label{cd91}
\end{eqnarray}
\end{lemma}
\begin{proof}
Since (\ref{ncd16}) holds then 
${\bf w}$ satisfies the diffusion equation of (\ref{cd1}). By using the weak formulation (\ref{cd3}), divergence theorem and the boundary conditions(\ref{cd1b}),  we have
\begin{eqnarray*}
\frac{1}{2}\frac{d}{dt}\int_{\Omega}(w_{1}^{2}+w_{2}^{2})d\Omega &=&-\int_{\Omega}\mathrm{tr}\left((J{\bf w})^{T}D(J{\bf w})\right)d\Omega.
\end{eqnarray*}
Now, (H1) and (\ref{cd3}) imply
\begin{eqnarray*}
\mathrm{tr}\left((J{\bf w})^{T}D(J{\bf w})\right)\geq \alpha
||J{\bf w}||_{X_{0}}^{2}.
\end{eqnarray*}
Therefore
\begin{eqnarray*}
\frac{d}{dt}||{\bf w}||_{X_{0}}^{2}\leq -2\alpha ||J{\bf w}||_{X_{0}}^{2}.
\end{eqnarray*}
Note now that if we apply the Poincar\'e inequality to each $w_{i}, i=1,2$, then there is $C_{0}>0$ such that
\begin{eqnarray*}
||{\bf w}||_{X_{0}}^{2}\leq C_{0} ||J{\bf w}||_{X_{0}}^{2}.
\end{eqnarray*}
This implies that
\begin{eqnarray*}
\frac{d}{dt}||{\bf w}||_{X_{0}}^{2}\leq -2\alpha C_{0} ||{\bf w}||_{X_{0}}^{2}.
\end{eqnarray*}
By Gronwall's lemma
\begin{eqnarray*}
||{\bf w}(t)||_{X_{0}}^{2}\leq e^{-2\alpha C_{0}t} ||{\bf w}(0)||_{X_{0}}^{2}
\end{eqnarray*}
and (\ref{cd91}) holds.
\end{proof}
\begin{remark}
Since we are assuming strong ellipticity (H1) in the model, the asymptotic behaviour (\ref{cd91}) is expected. In particular, we conclude that the model does not preserve the discontinuities of the initial conditions, which is a serious limitation in the Computer Vision context. Assumption (H1) could be relaxed by considering degenerate elliptic cross-diffusion operators $D$, that is, substituting (\ref{cd2}) by
\begin{eqnarray*}
{\bf \xi}^{T}D({\bf u}({\bf x},t)){\bf \xi}\geq 0,\quad ({\bf x},t)\in \overline{Q_{T}}.
\end{eqnarray*} 
The analysis of such models is beyond the scope of this work.
\end{remark}

\section{Numerical experiments}
\label{sec:sec3}
The performance of (\ref{cd1})-(\ref{cd1b}) in filtering problems is numerically illustrated in this section.
\subsection{The numerical procedure}
\label{sec:sec31}
In order to implement (\ref{cd1})-(\ref{cd1b}) some details are described. The first point concerns the choice of the cross-diffusion matrix $D$. We have considered to this end the results on linear cross-diffusion shown in the companion paper Ara\'ujo et al. \cite{ABCD2016} and the complex diffusion approach, Gilboa et al. \cite{GilboaSZ2004}, see Remark \ref{rem1}. According to them, matrices of the form
\begin{eqnarray}
D(u,v)=g(|w|)d,\quad d=\begin{pmatrix} d_{11}&d_{12}\\d_{21}&d_{22}\end{pmatrix},\label{cd31}
\end{eqnarray}
were used for the experiments, with
\begin{eqnarray}
g(v)=\frac{1}{1+\left(\frac{v}{\kappa}\right)^2},\label{cd32}
\end{eqnarray}
with $\kappa$ a threshold parameter, see Gilboa et al. \cite{GilboaSZ2004} and $d$ a positive definite matrix. Both possibilities $w=M(v)$ in (\ref{cdr3}) and $w=w_{\sigma}$ in (\ref{cdr4}) have been implemented. The form (\ref{cd31}), (\ref{cd32}) takes into account (\ref{cdr1}) by using the extended version of the small theta approximation, see Ara\'ujo et al. \cite{ABCD2016} (which justifies the presence of $d_{12}$ in (\ref{cd32})) as well as the classical nonlinear diffusion approach with the form of $g$, see Aubert \& Kornprobst, \cite{AubertK2001}, Catt\'e et al. \cite{CatteLMC1992}, Perona \& Malik \cite{PeronaM1990}.

The guidance about the choice of the matrix $d$ was also based on linear cross-diffusion. Thus if $s=(d_{22}-d_{11})^{2}+4d_{12}d_{21}$, three types of matrices $d$ (for which $s>0$, $s<0$ and $s=0$) have been considered. (The parameter $s$ determines if the eigenvalues of $d$ are real or complex, see Ara\'ujo et al. \cite{ABCD2016}.) The specific examples of $d$ for the experiments are given in Section \ref{sec:sec32}.

A second question on the implementation concerns the choice of a numerical scheme to approximate (\ref{cd1})-(\ref{cd1b}). Thus, the explicit numerical method introduced and analyzed in Ara\'ujo et al. \cite{AraujoBS2012} for the complex diffusion case has been adapted here. The method is now briefly described. By using the notation of Section \ref{sec:sec2}, $\overline{Q}_{T}$ is first dicretized as follows. We define a uniform grid on $\overline{\Omega}=[l_{1},r_{1}]\times [l_{2},r_{2}]$ with mesh step size $h>0$ as
\begin{eqnarray}
\overline{\Omega}_{h}&=&\{{\bf x}_{ij}=(x_{i},y_{j})\in \overline{\Omega}: x_{i}=l_{1}+ih, y_{j}=l_{2}+jh,\nonumber\\
&&i=0,...,N_1-1,
j=0,...,N_2-1\},\label{cd33}
\end{eqnarray}
for integers $N_{1}, N_{2}>1$ such that $hN_{i}=r_{i}-l_{i}, i=1,2$. As far as the time discretization is concerned, fixed $T>0$, for an integer $M\geq 1$ and $\Delta>0$ such that $M\Delta t =T$, the interval $[0,T]$ is partitioned in 
\begin{eqnarray}
0=t^0<t^1<\cdots<t^{M-1}<t^M=T,\label{cd34}
\end{eqnarray}
with $t^{m+1}=t^{m}+\Delta t, m=0,\ldots,M-1$. The resulting discretization of $\overline{Q}_{T}$ with (\ref{cd33}) and (\ref{cd34}) is denoted by $\overline{Q}_h^{\Delta t}={Q}_h^{\Delta t}\cup {\Gamma}_h^{\Delta t}$, where ${Q}_h^{\Delta t}, {\Gamma}_h^{\Delta t}$ stand for the interior and boundary meshes, respectively.

If $M_{N_{1}\times N_{2}}(\mathbb{R})$ denotes the space of $N_{1}\times N_{2}$ real matrices then let us consider some initial distribution $U^{0}, V^{0}: \overline{\Omega}_{h}\rightarrow M_{N_{1}\times N_{2}}(\mathbb{R})$ with $U^{0}=(U_{ij}^{0})_{i=1,j=1}^{N_{1},N_{2}}, V^{0}=(V_{ij}^{0})_{i=1,j=1}^{N_{1},N_{2}}$. From $U^{0}, V^{0}$ and for $m=0,\ldots,M-1$ the approximate image at time $t^{m+1}$ is defined as $(U^{m+1},V^{m+1})^{T}$, where $U^{m+1}=(U_{ij}^{m+1})_{i=1,j=1}^{N_{1},N_{2}}, V^{m+1}=(V_{ij}^{m+1})_{i=1,j=1}^{N_{1},N_{2}}: \overline{\Omega}_{h}\rightarrow M_{N_{1}\times N_{2}}(\mathbb{R})$ satisfy the system (in vector form)
\begin{eqnarray}
\displaystyle  \frac{U^{m+1}-U^m}{\Delta t} & = & \nabla_h\cdot (g(V^m)(d_{11}\nabla_h U^m +d_{12} \nabla_h V^m)),\nonumber\\
\displaystyle  \frac{V^{m+1}-V^m}{\Delta t} & = & \nabla_h\cdot (g(V^m)(d_{21}\nabla_h U^m +d_{22}\nabla_h V^m)),\nonumber\\
&&\label{cd35}
\end{eqnarray}
where $g$ and $d$ are given by (\ref{cd31}), (\ref{cd32}). In (\ref{cd35}), $\nabla_h$ is the discrete operator such that if $W=(W_{ij})_{i=1,j=1}^{N_{1},N_{2}}$ then
\begin{eqnarray*}
\left(\nabla_h W\right)_{ij}&=&\left(\frac{W_{i+1,j}-W_{i-1,j}}{2h},\frac{W_{i,j+1}-W_{i,j-1}}{2h}\right),\\
&& i=1,\ldots,N_{1}-1, j=1,\ldots,N_{2}-1.
\end{eqnarray*}
The scheme (\ref{cd35}) is completed with the discretization of the Neumann boundary conditions (\ref{cd1a}) by using $\nabla_{h}$.

We note that for the complex diffusion case, a stability condition for (\ref{cd35}) with diffusion coefficient (\ref{cdr1}) was derived in Ara\'ujo et al. \cite{AraujoBS2012} and Bernardes et al. \cite{Bernardes:10},
\begin{eqnarray}
\Delta t:=\max_{0\leq m\leq M-1} \Delta t^{m} \le \frac{\cos{\theta}}{4}\left(1+\frac{\min_m{(V^m)^2}}{{\kappa}^2\theta^2} \right).\label{cd36}
\end{eqnarray}
Condition (\ref{cd36}) was taken into account in the numerical experiments below, where $h=1$ and $\Delta t=0.05$ (with $\kappa=10$) were used.
\subsection{Numerical results}
\label{sec:sec32}
In this section several numerical experiments have been performed according to the following steps. Assume that the discrete values $S_{ij}=S({\bf x}_{ij}), {\bf x}_{ij}\in \overline{\Omega}_{h}$ of some real-valued function $S:\overline{\Omega}:\rightarrow\mathbb{R}$ represent an original image on $\overline{\Omega}_{h}$. From $S$, some noise of Gaussian type with zero mean and standard deviation $\sigma^{\prime}$ at pixel ${\bf x}_{ij}$ is added to $S_{ij}$. This is represented by a matrix $N(\sigma^{\prime})=(N_{ij}(\sigma^{\prime}))_{ij}$  and generates the initial noisy image values
\begin{eqnarray*}
{u}^{0}&=&(u_{ij}^{0})_{ij}, \quad u_{ij}^{0}=S_{ij}+N_{ij}(\sigma^{\prime}), \\
 &&i=1,\ldots,N_{1}, j=1,\ldots,N_{2}.
\end{eqnarray*}
Then the explicit method (\ref{cd35}) with initial distribution $U_{ij}^{0}=u_{ij}^{0}, V_{ij}^{0}=0, i=1,\ldots,N_{1}, j=1,\ldots,N_{2}$ and $D$ from (\ref{cd31}), (\ref{cd32}) is run; the corresponding numerical solutions $U^{m}, V^{m}, m=1,\ldots,M$ are monitored in such a way that $U^{m}$ approximates the original signal $S$ at time $t^{m}$. In order to measure the quality of restoration, three metrics are used:
\begin{itemize}
\item Signal-to-Noise-Ratio (SNR):
\begin{eqnarray}
SNR (S,U^m) = 10\log_{10}\left( \frac{\Var(S)}{\Var(U^m-S)}\right),\label{snr}
\end{eqnarray}
where the variance ($\Var$) of an image $U$ is defined by
$$\Var(U) = \frac{1}{N_1N_2}\|U-\bar{U}\|^2_F,$$
$\|\cdot\|_F$ stands for the Frobenius norm and $\bar{U}$ is a uniform image with intensities equal to the mean value of the intensities of $U$.
\item Peak Signal-to-Noise-Ratio (PSNR):
 \begin{eqnarray}
PSNR(S,U^m) = 20\log_{10}\left( \frac{255}{RMSE(S,U^m)}\right),\label{psnr}
\end{eqnarray}
where the Root-Mean-Square-Error ($RMSE$) is defined as
$$RMSE(S,U^{m})=\frac{1}{\sqrt{N_1N_2}}\|S-U^{m}\|_F^2;$$
\item The no-reference perceptual blur metric (NPB) proposed by Cr\'et\'e-Roffet et al. \cite{CRDLN2007}. This is based on evaluating the blur annoyance of the image by comparing the variations between neighbouring pixels before and after the application of a low-pass filter. The estimation ranges from $0$ (the best quality blur perception) to $1$ (the worst one).
\end{itemize}
\begin{figure}[htbp]
\centering
\subfigure[]
{\includegraphics[width=8.2cm,height=8.2cm]{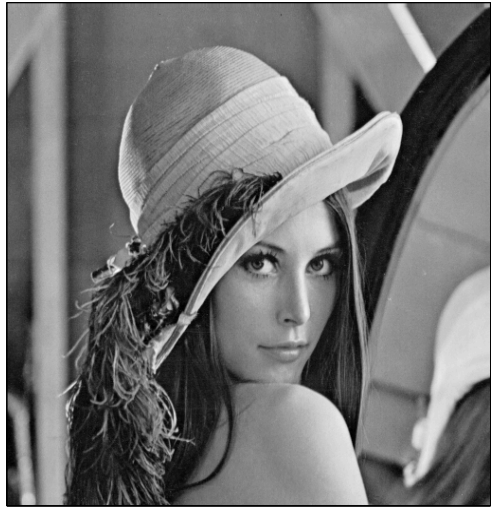}}
\subfigure[]
{\includegraphics[width=8.2cm,height=8.2cm]{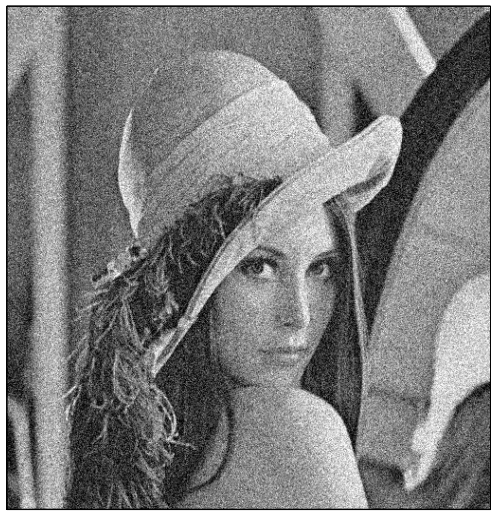}}
\caption{(a) Original image $S$ of Lena; (b) Noisy image of Lena with Gaussian noise of $\sigma^{\prime}=30$.}
\label{figRII_1}
\end{figure}
The following numerical results illustrate the behaviour of (\ref{cd1})-(\ref{cd1b}) according to the choice of the matrix $d$ in (\ref{cd31}) and the implementation of (\ref{cd32}). The experiments are concerned with the filtering of a noisy image of Lena (Figure \ref{figRII_1}) and a first group makes use of the matrices
\begin{itemize}
\item[(i)] NCDF1 ($s>0$):
$d_{11}=1, d_{12}=0.025, d_{21}=1, d_{22}=1$.
\item[(ii)] NCDF2 ($s<0$):
$d_{11}=1, d_{12}=-0.025, d_{21}=0.025,$ $d_{22}=1$.
\item[(iii)] NCDF3 ($s=0$):
$d_{11}=1, d_{12}=-0.025, d_{21}=1, d_{22}=1.1$.
\end{itemize}
 These models were taken to study three points of the filtering: the restoration process from the first component of the numerical solution of (\ref{cd35}), the behaviour of the edges from the second component and the quality of filtering from the computation of the evolution of the three metrics. 
\begin{figure}[htbp]
\centering
\subfigure[]
{\includegraphics[width=0.5\textwidth]{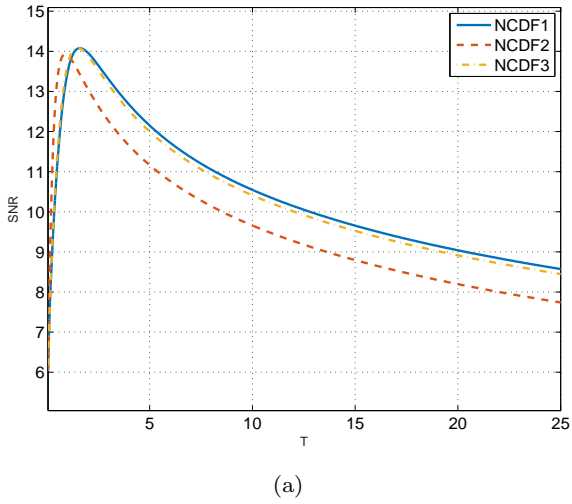}}
\subfigure[]
{\includegraphics[width=0.5\textwidth]{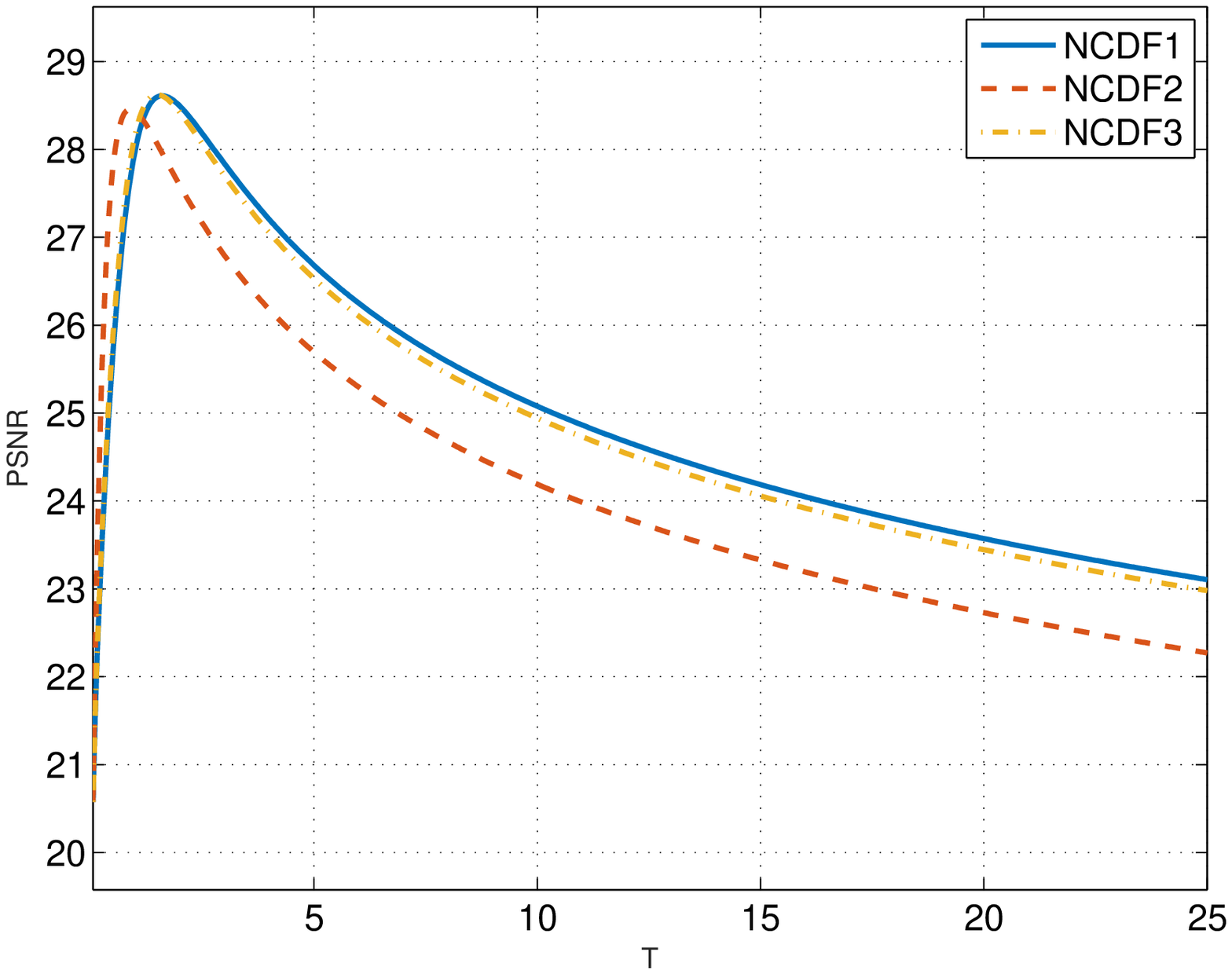}}
\caption{$SNR$ (a) and $PSNR$ (b)  vs. time: NCDF1 (solid line), NCDF2 (dashed line) and NCDF3 (dashed-dotted line).}
\label{figRII_2}
\end{figure}
The numerical experiments  in Figure \ref{figRII_2},  show the time evolution of the SNR and PSNR parameters given by the models NCDF1-3. For the three models, the metrics attain a maximum value from which the quality of restoration is decreasing. The main difference appears in the time at which the maximum holds, being longer in the case of NCDF1 (corresponding to $s>0$) and NCDF3 (for which $s=0$) then in the model NCDF2 (where $s<0$: this would illustrate the complex diffusion case, see Ara\'ujo et al. \cite{ABCD2016}). Note also from Figure \ref{figRII_2} that NCDF1 and NCDF3 will provide a better evolution of the two metrics: they will be more suitable than NCDF2 for long time restoration processes, while NCDF2 performs better in short computations. 
\begin{figure}
\centering
\subfigure[]
{\includegraphics[width=8.2cm,height=8.2cm]{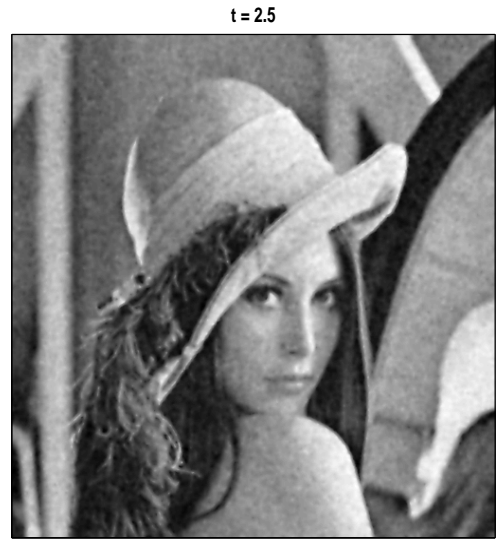}}
\subfigure[]
{\includegraphics[width=8.2cm,height=8.2cm]{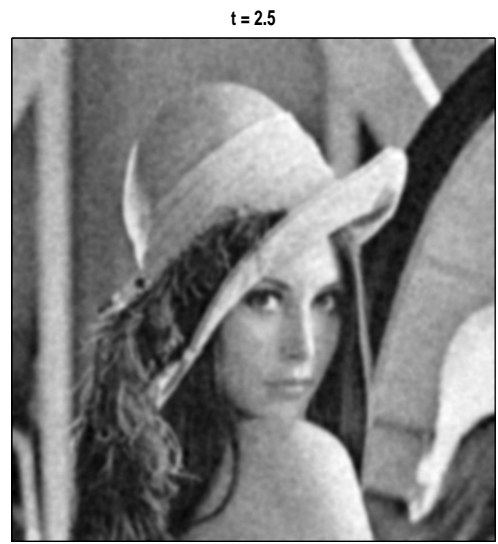}}
\caption{First component of solution of (\ref{cd35}) at time $t=2.5$ with (a) NCDF1 and (b) NCDF2.}
\label{figRII_3a}
\end{figure}
\begin{figure}[htbp]
\centering
\subfigure[]
{\includegraphics[width=8.2cm,height=8.2cm]{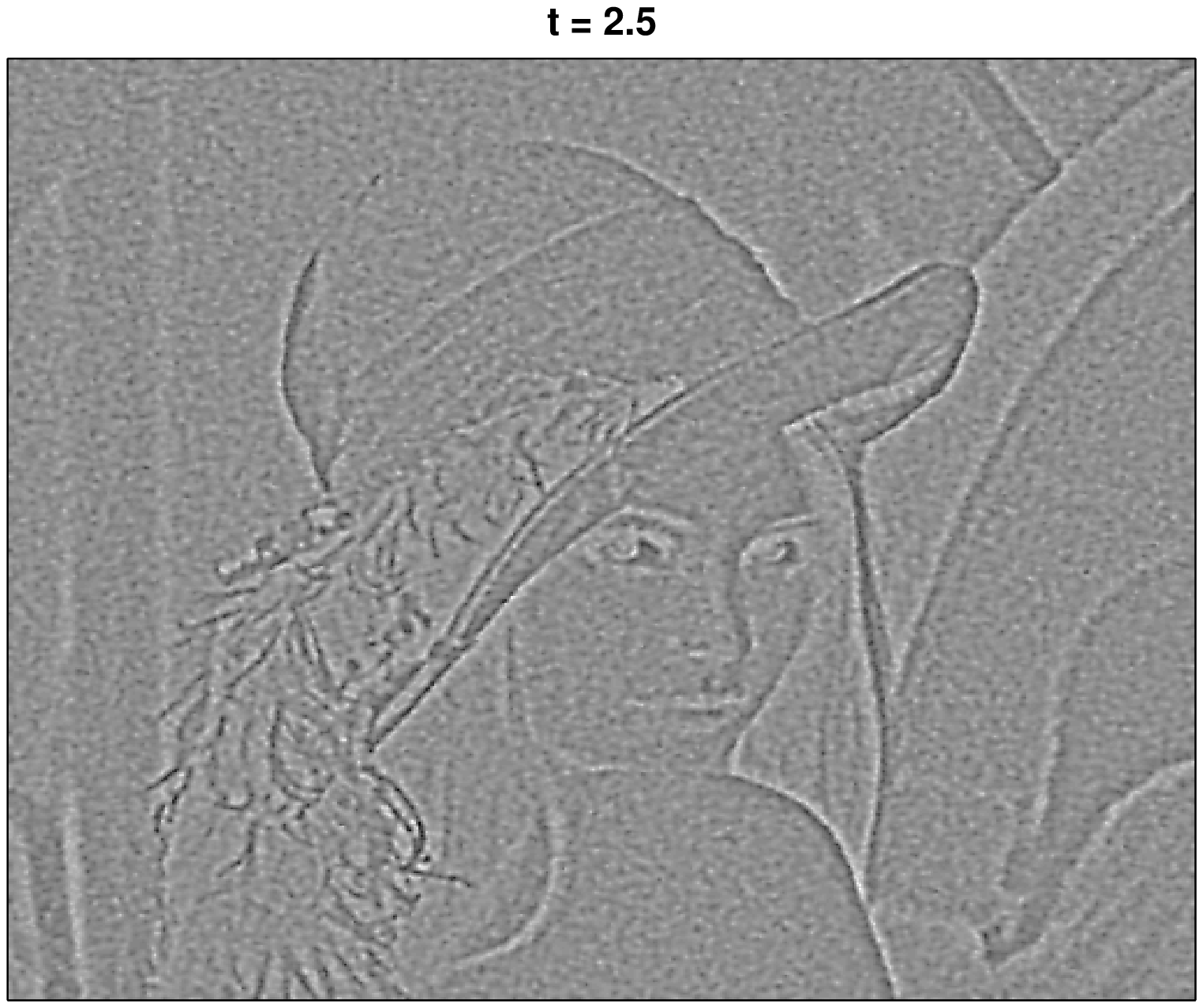}}
\subfigure[]
{\includegraphics[width=8.2cm,height=8.2cm]{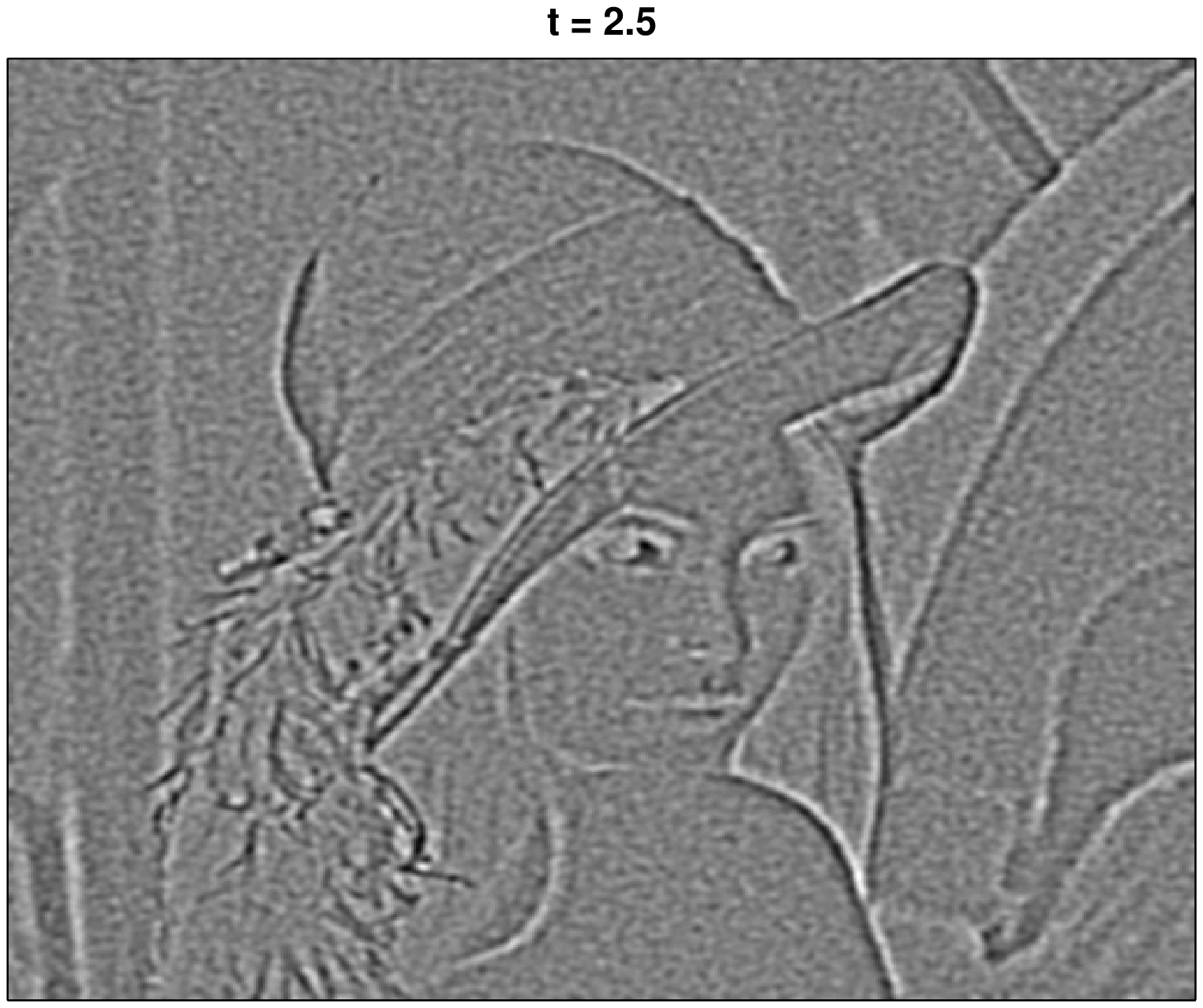}}
\caption{Second component of solution of (\ref{cd35}) at time $t=2.5$ with (a) NCDF1 and (b) NCDF2.}
\label{figRII_3b}
\end{figure}
Since the longer the evolution the more noise is removed, models NCDF1 and NCDF3 suggest a better control of the diffusion to improve the quality of the restored images. This is observed in Figures \ref{figRII_3a} and \ref{figRII_3b}, which show the two components of the solution of (\ref{cd35}) at time $t=2.5$ given by NCDF1 and NCDF2. (The images corresponding to NCDF3 are similar to those of NCDF1 and will not be shown here.) Observe that the second component has the role of edge detector and it is less affected by noise and over diffusion in the case of NCDF1.
\begin{figure}[htbp]
\centering
\subfigure[]
{\includegraphics[width=0.5\textwidth]{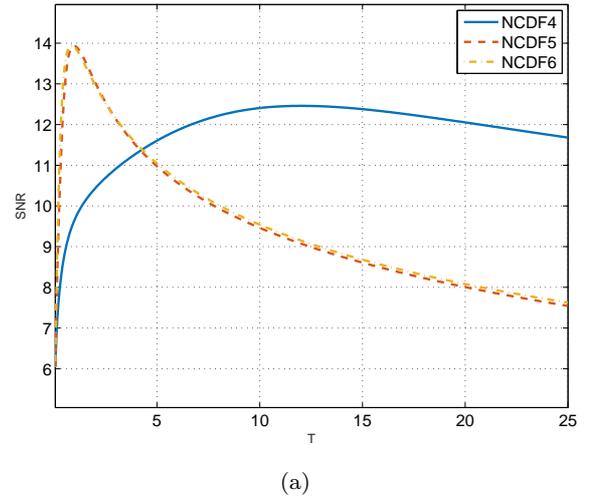}}
\subfigure[]
{\includegraphics[width=0.5\textwidth]{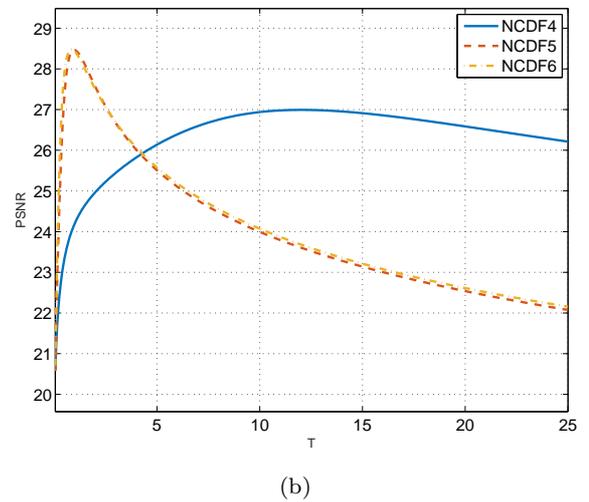}}
\caption{$SNR$ (a) and $PSNR$ (b)  vs. time: NCDF4 (solid line), NCDF5 (dashed line) and NCDF6 (dashed-dotted line).}
\label{figRII_4}
\end{figure}
For large values in magnitude of the entries of the matrix $d$ in (\ref{cd32}) the differences in the models are more significant. This is illustrated by a second group of experiments, for which the matrices are
\begin{itemize}
\item[(i)] NCDF4 ($s>0$):
$d_{11}=1, d_{12}=0.9, d_{21}=1, d_{22}=1$.
\item[(ii)] NCDF5 ($s<0$):
$d_{11}=1, d_{12}=-0.9, d_{21}=0.9, d_{22}=1$.
\item[(iii)] NCDF6 ($s=0$):
$d_{11}=1, d_{12}=-0.9, d_{21}=0.225, 
d_{22}=1.9$,
\end{itemize} 
and the rest of the implementation data is the same as that of the previous experiments. The evolution of the SNR and PSNR values is now shown in Figure \ref{figRII_4}. Note that the behaviour of NCDF5 and NCDF6 is very similar and their quality metrics, compared to those of NCDF4, are more suitable up to a time of filtering close to $t=5$. From this time, NCDF4 behaves better and becomes a better choice to filter for a longer time. 
\begin{figure}[htbp]
\centering
\subfigure
{\includegraphics[width=8.2cm,height=8.2cm]{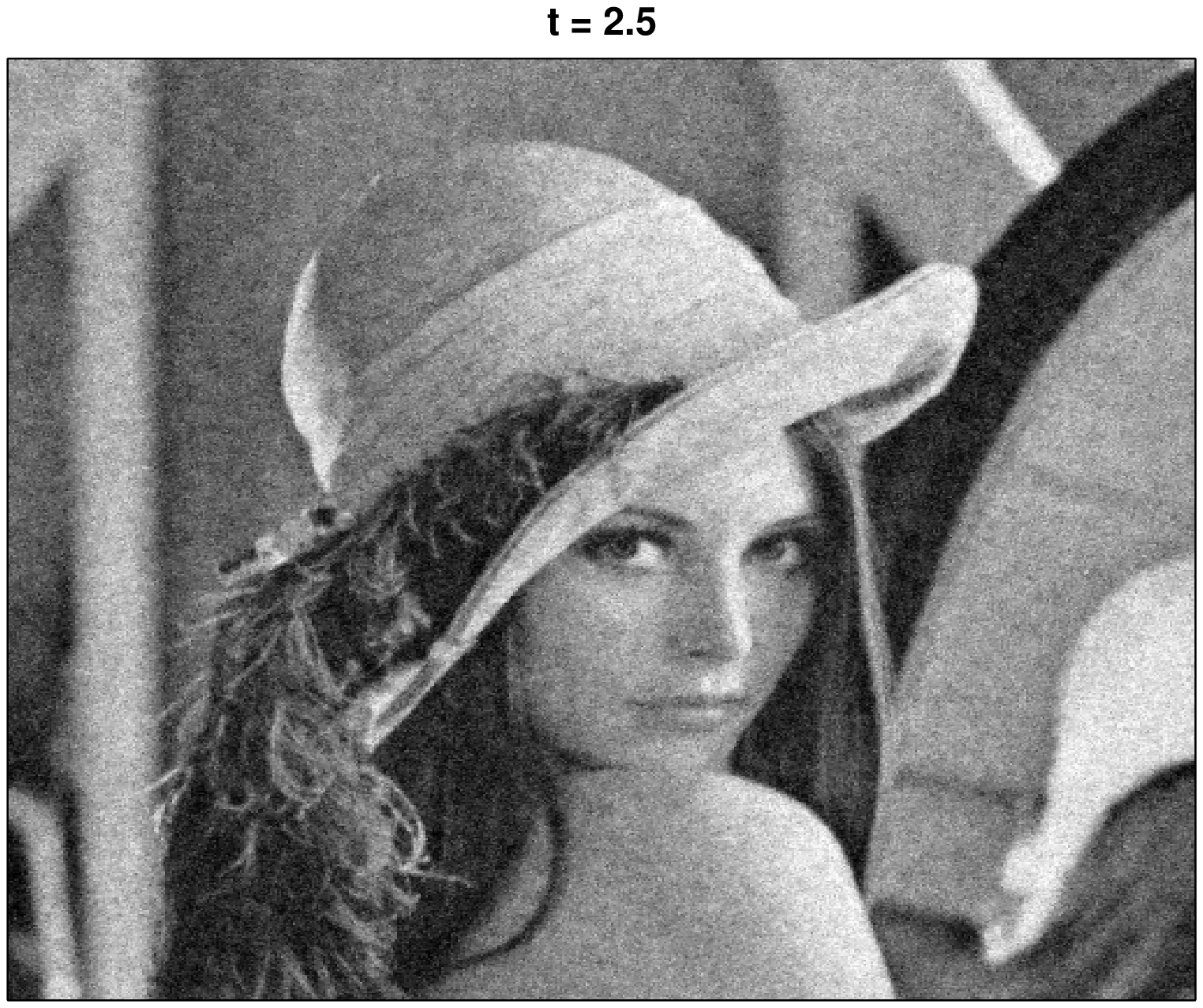}}
\subfigure
{\includegraphics[width=8.2cm,height=8.2cm]{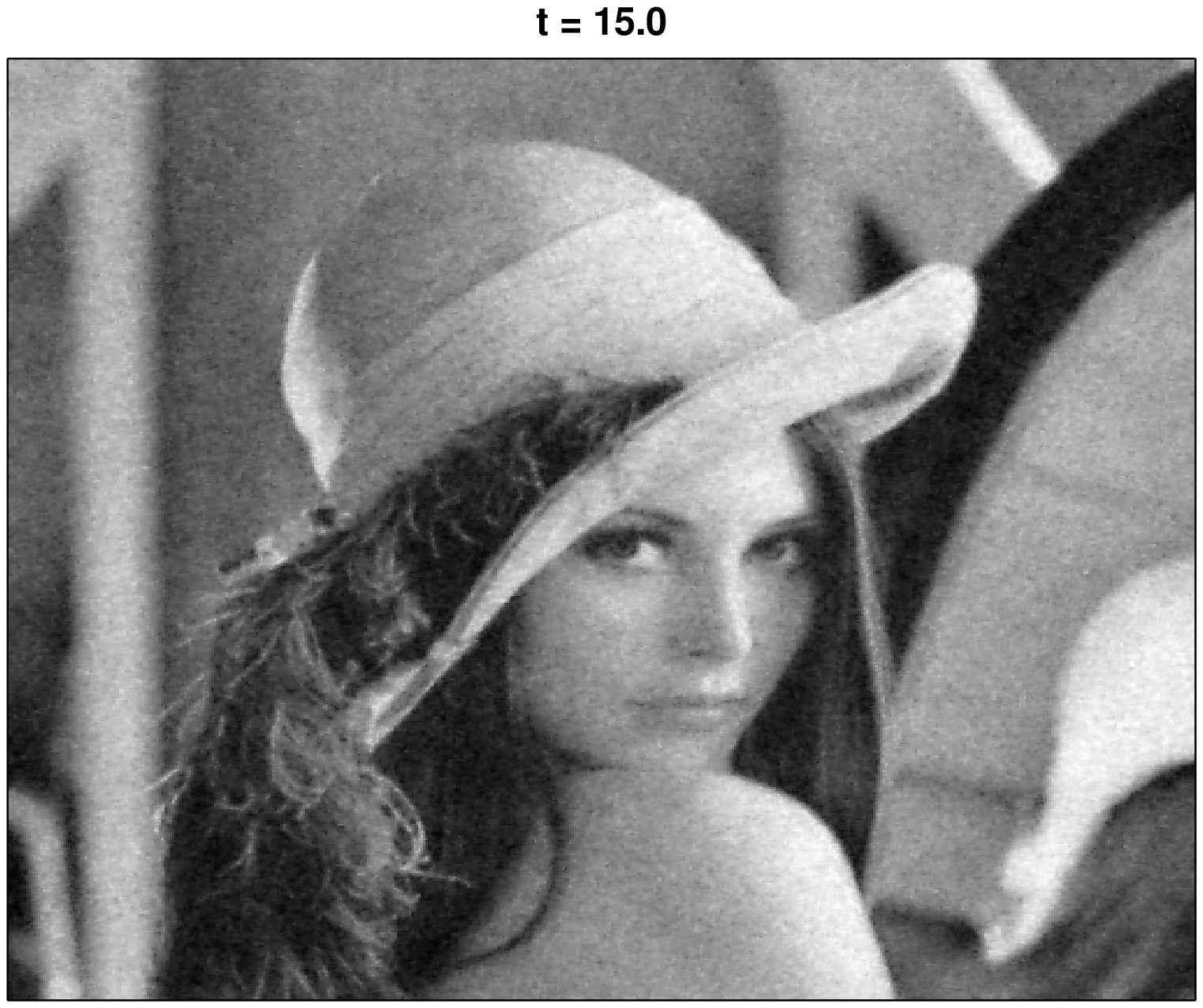}}
\subfigure
{\includegraphics[width=8.2cm,height=8.2cm]{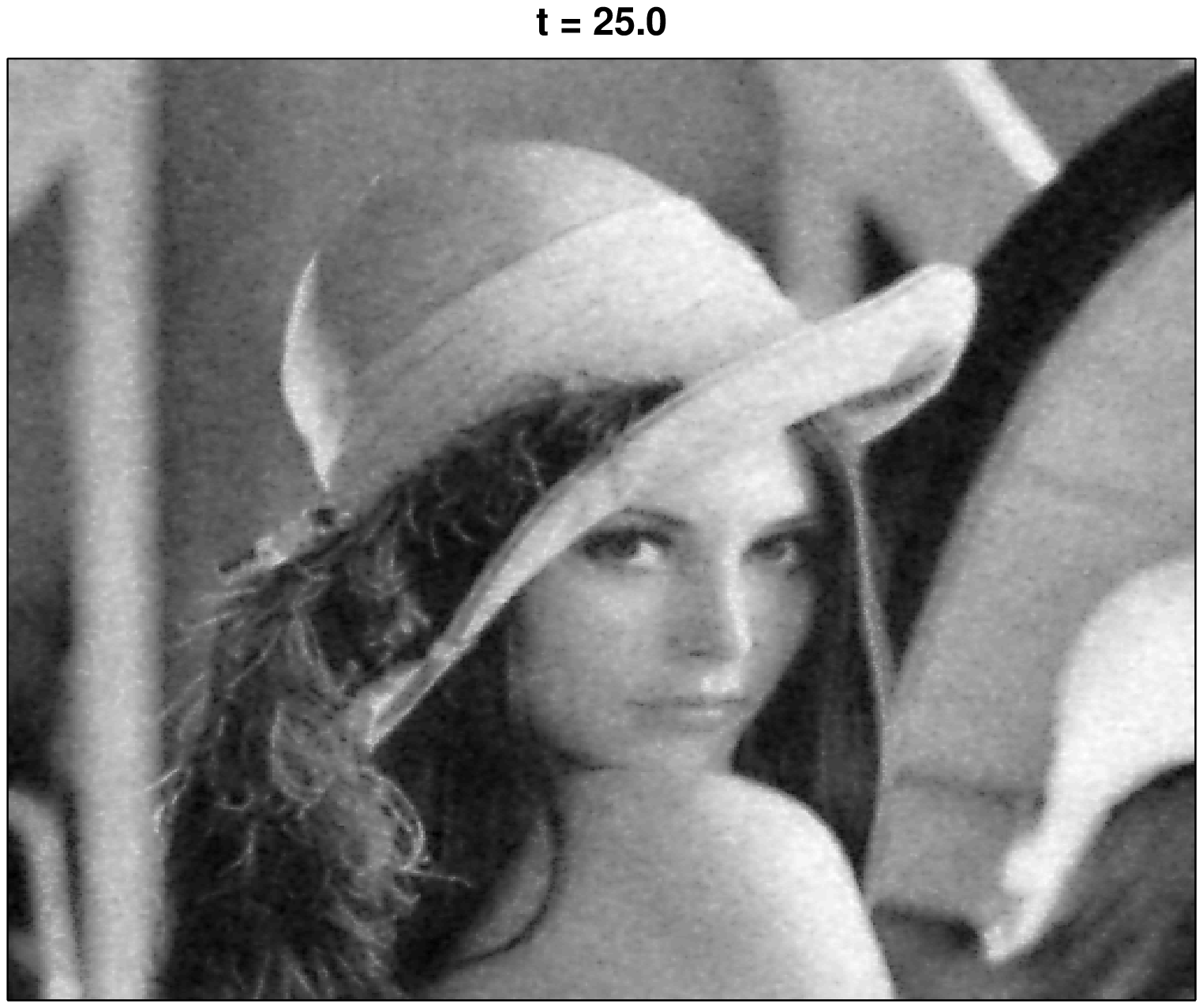}}
\caption{First component of the solution of (\ref{cd35}) at times $t=2.5, 15, 25$ with NCDF4.}
\label{figRII_5a}
\end{figure}
\begin{figure}[htbp]
\centering
\subfigure
{\includegraphics[width=8.2cm,height=8.2cm]{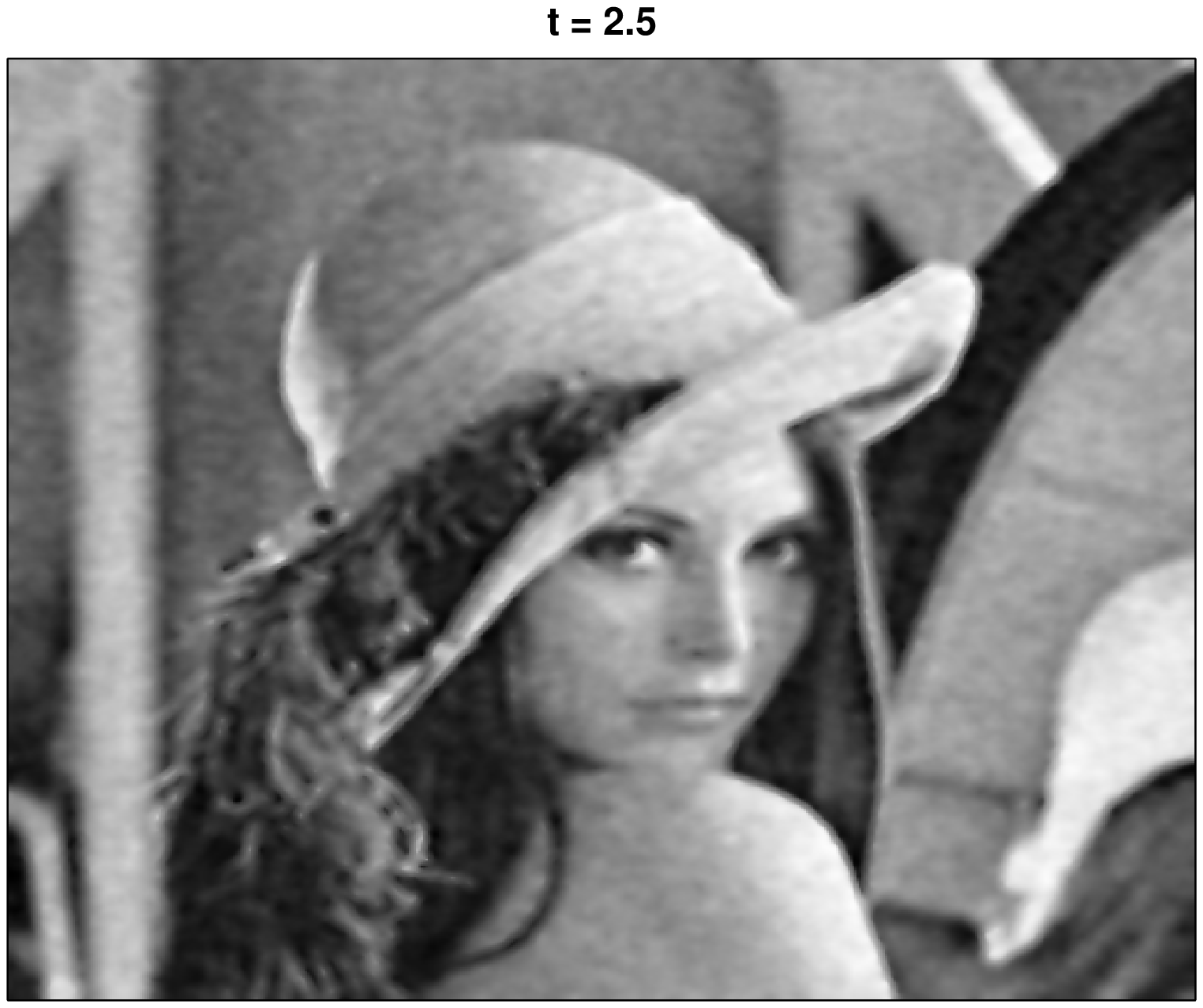}}
\subfigure
{\includegraphics[width=8.2cm,height=8.2cm]{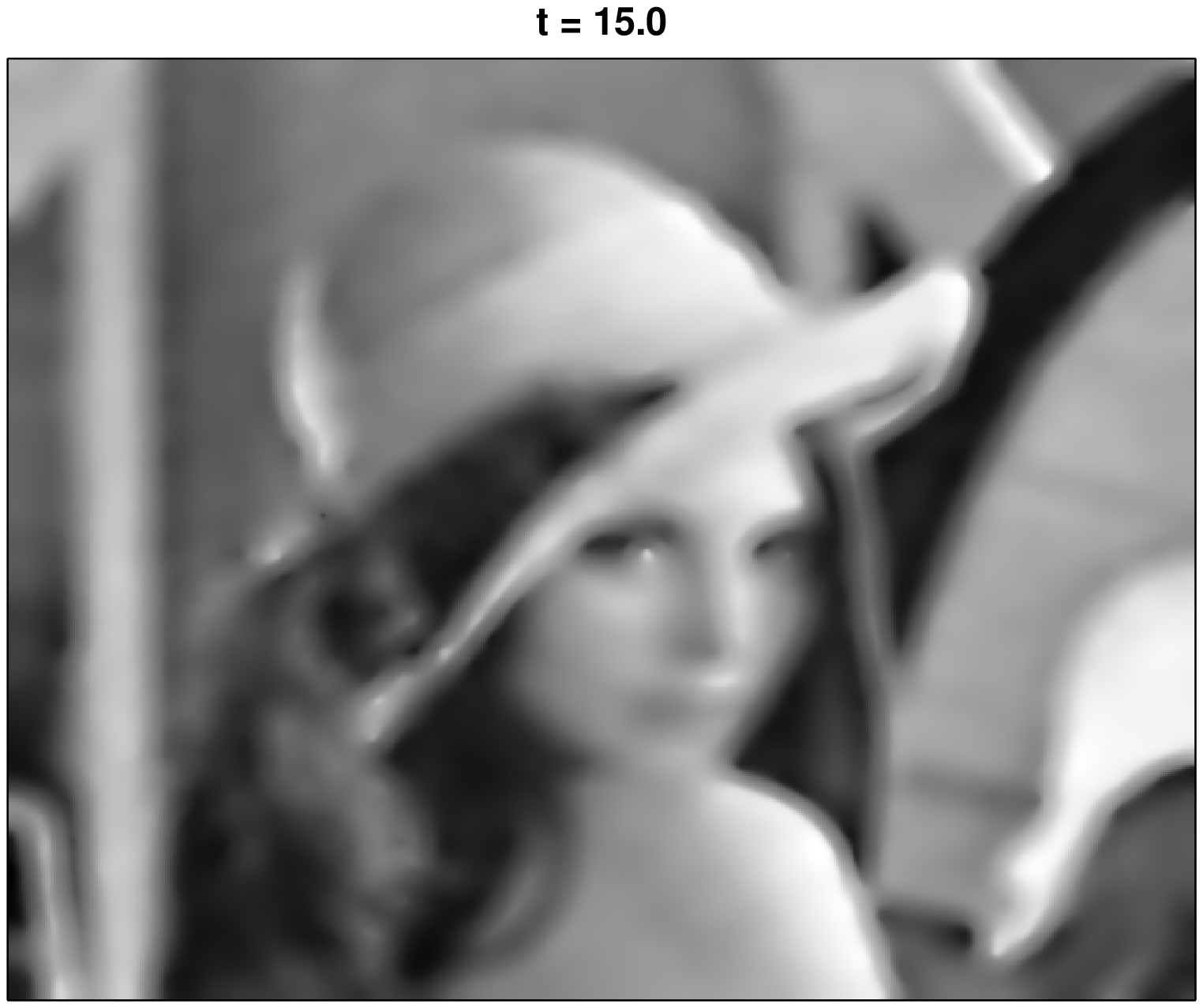}}
\subfigure
{\includegraphics[width=8.2cm,height=8.2cm]{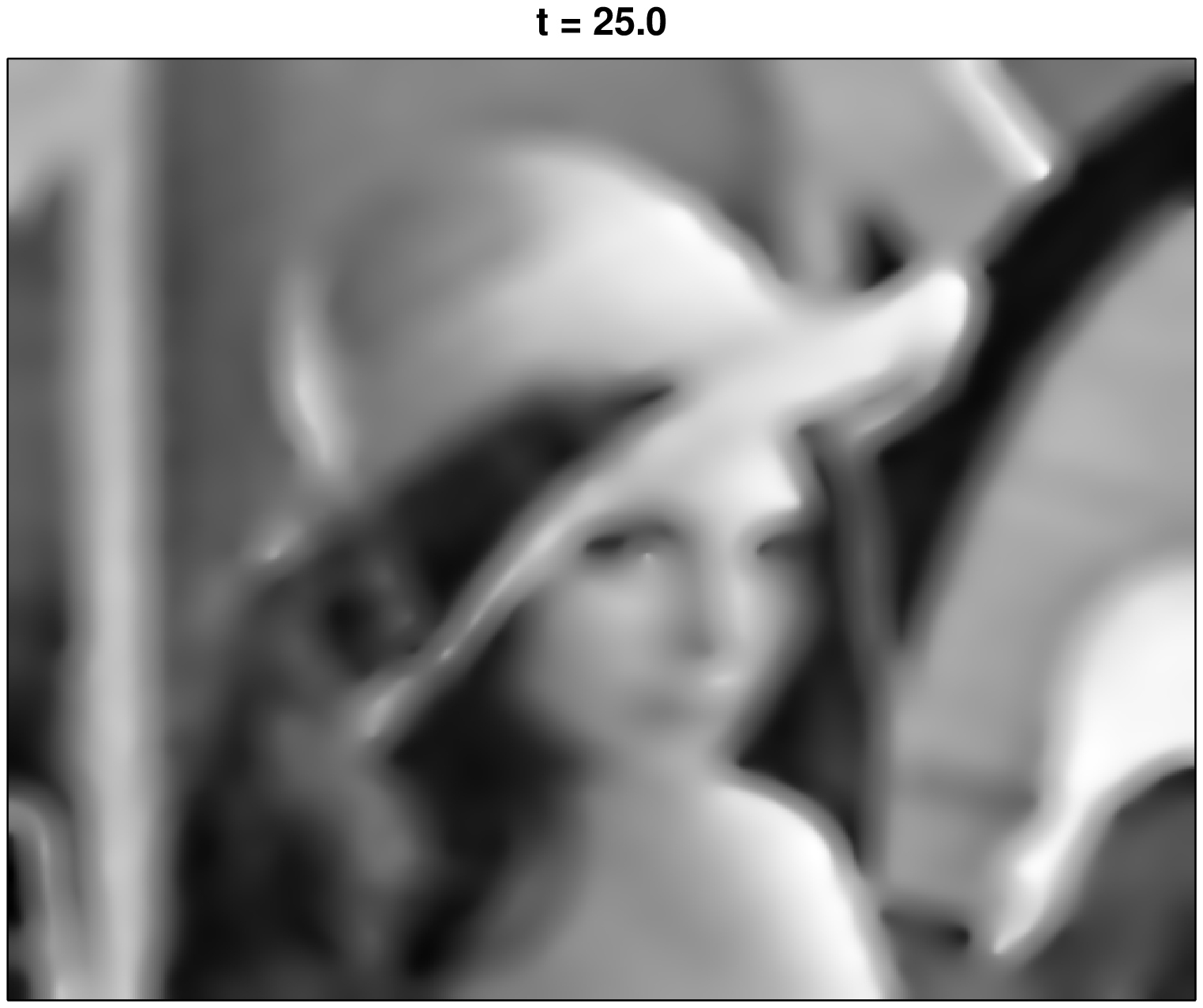}}
\caption{First component of the solution of (\ref{cd35}) at times $t=2.5, 15, 25$ with NCDF5.}
\label{figRII_5b}
\end{figure}
The comparison between the solutions of (\ref{cd35}) with the three models reveals these differences in a significant way, see Figures \ref{figRII_5a}-\ref{figRII_6b}, where the images corresponding to NCDF4 and NCDF5 at several times are displayed. (The results with NCDF6 are very similar to those of NCDF5.) In the case of the first component (Figures \ref{figRII_5a} and \ref{figRII_5b}), the performance of the models by $t=2.5$ are similar, but at longer times NCDF4 delays the blurring and leads to a restored image with better quality. This control of the diffusion is confirmed in {Figures \ref{figRII_7}-\ref{figRII_9}}, which show, for the three models, the time evolution of the NPB metric (right) and the corresponding first component of the solution of (\ref{cd35}) at the time for which the SNR value is maximum (left). (In each case this time corresponds to the iteration of the numerical scheme associated to the small circle in the figure on the right.) The reduction of the edge spreading is also observed in the detection of the edges by using the second components, see Figures \ref{figRII_6a} and \ref{figRII_6b}. The evolution of the NPB curve for NCDF4 implies the best quality in terms of blur perception, among the three models.
\begin{figure}[htbp]
\centering
\subfigure
{\includegraphics[width=8.2cm,height=8.2cm]{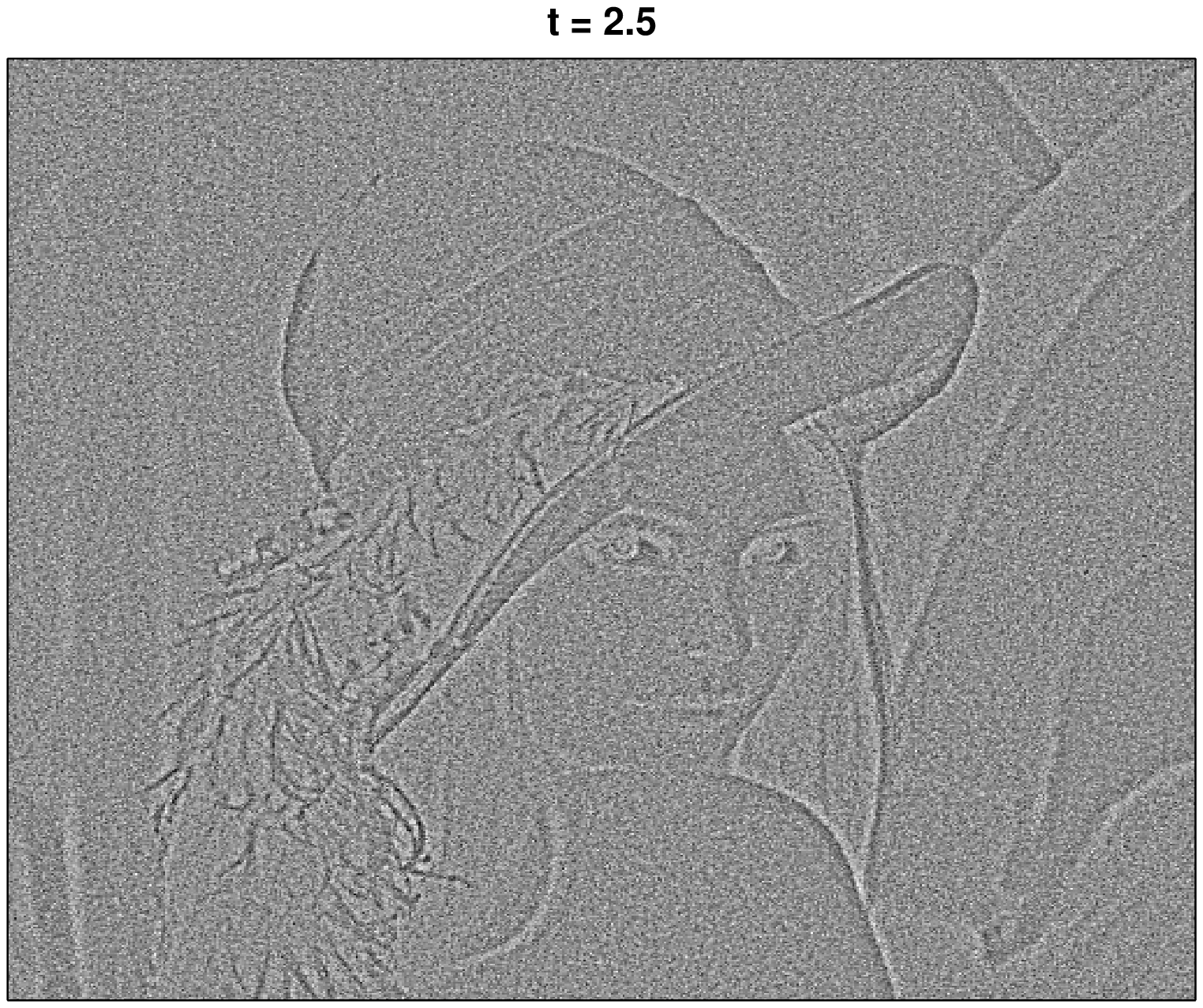}}
\subfigure
{\includegraphics[width=8.2cm,height=8.2cm]{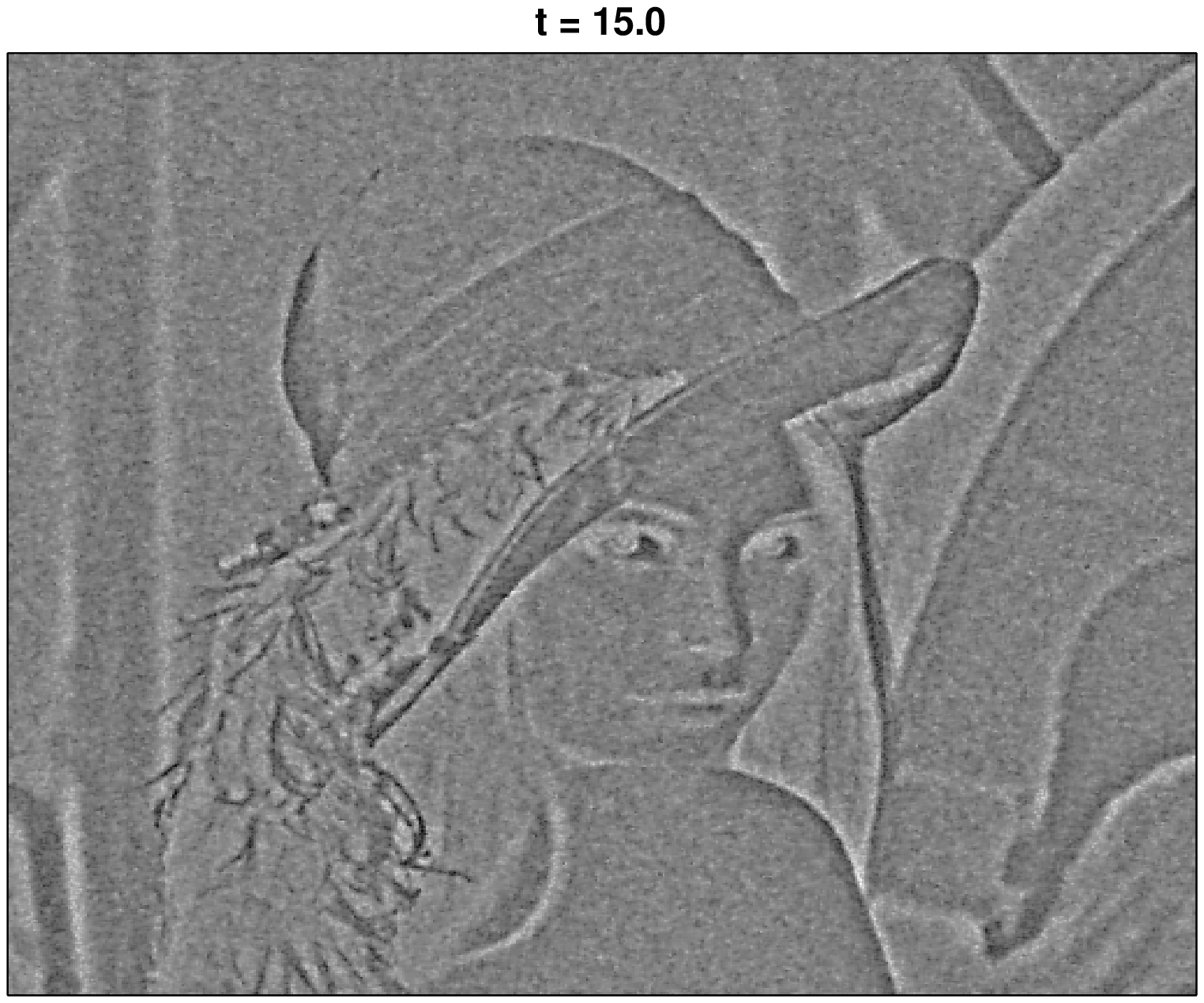}}
\subfigure
{\includegraphics[width=8.2cm,height=8.2cm]{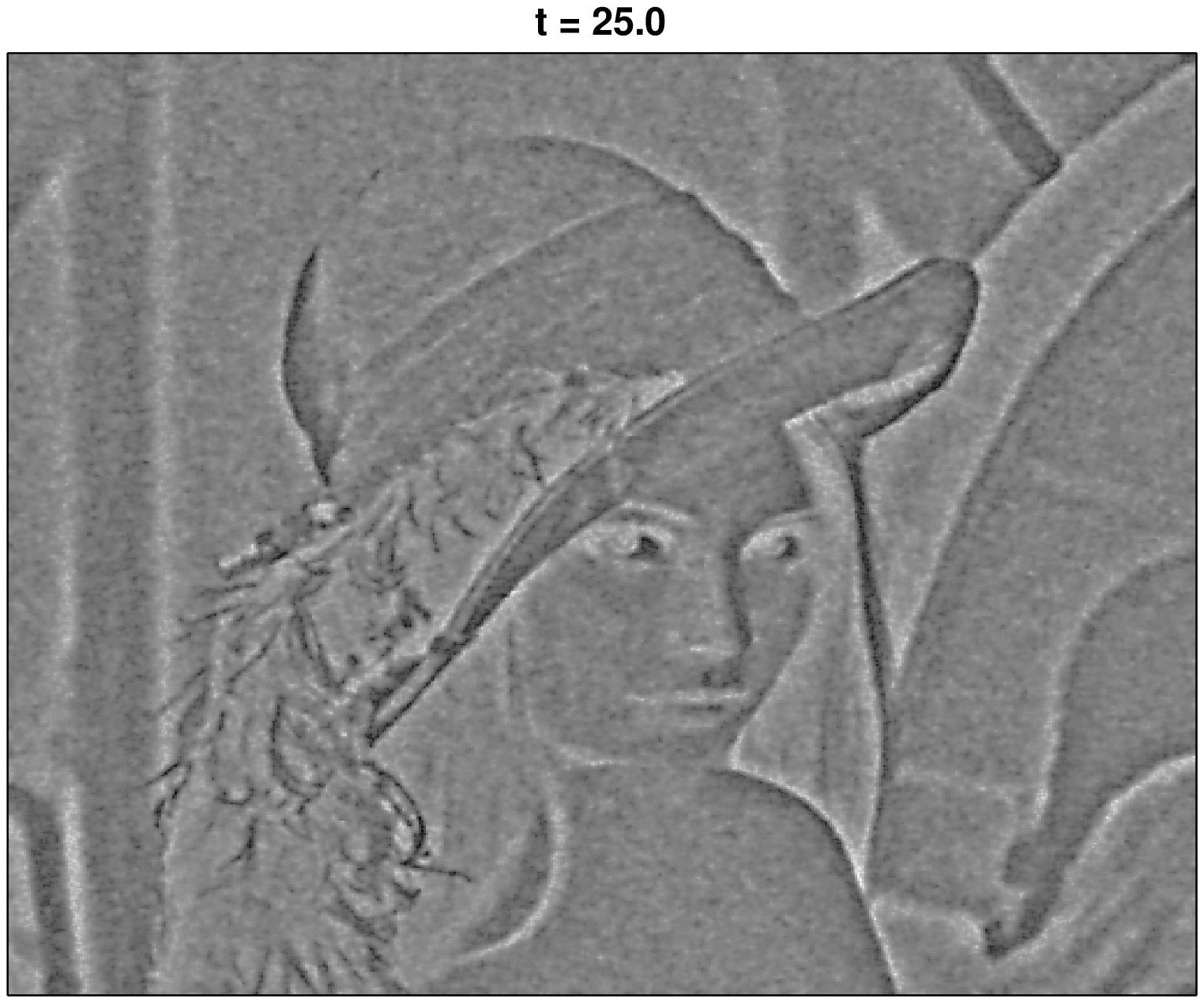}}
\caption{Second component of solution of (\ref{cd35}) at times $t=2.5, 15, 25$ with NCDF4.}
\label{figRII_6a}
\end{figure}
\begin{figure}[htbp]
\centering
\subfigure
{\includegraphics[width=8.2cm,height=8.2cm]{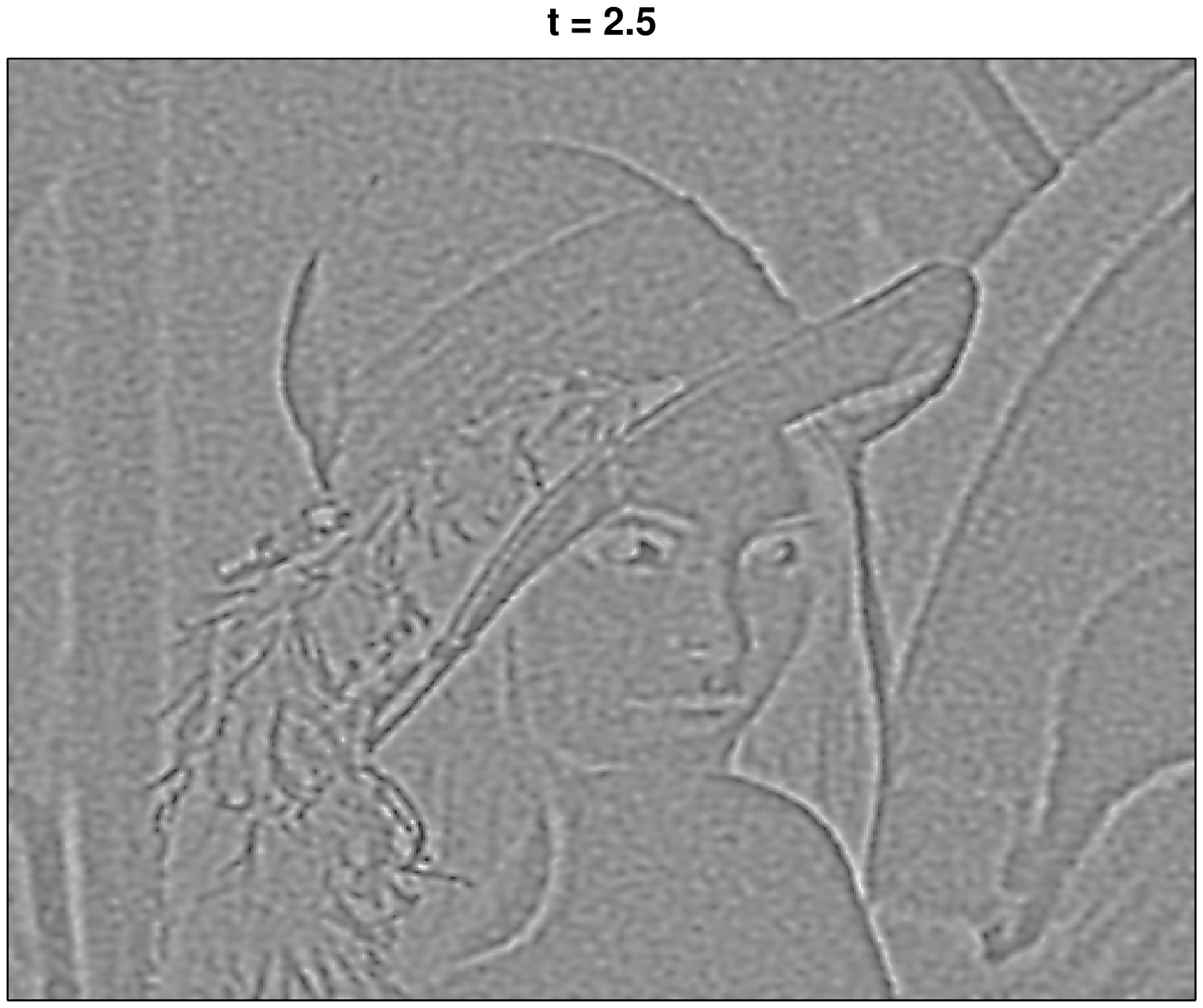}}
\subfigure
{\includegraphics[width=8.2cm,height=8.2cm]{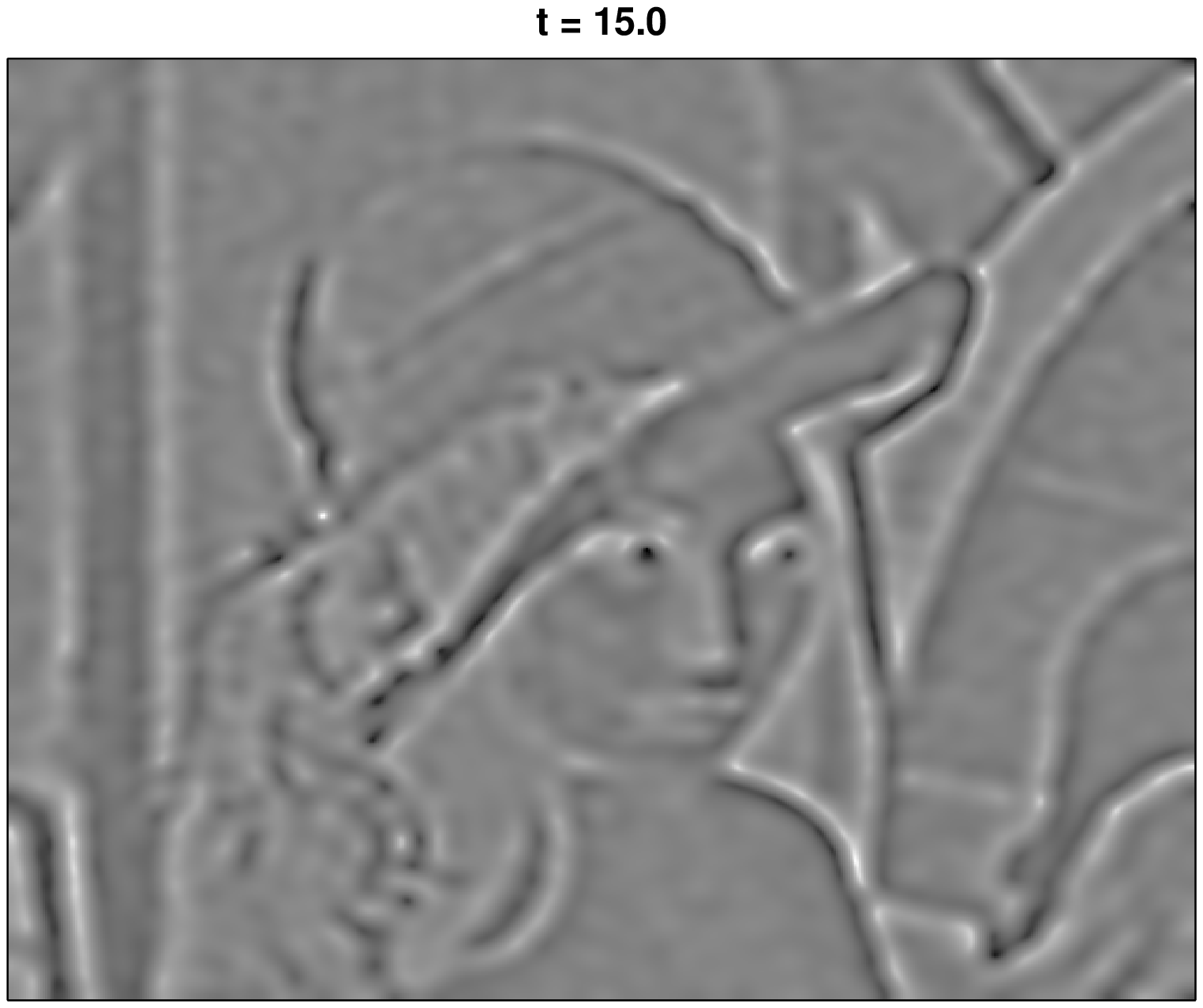}}
\subfigure
{\includegraphics[width=8.2cm,height=8.2cm]{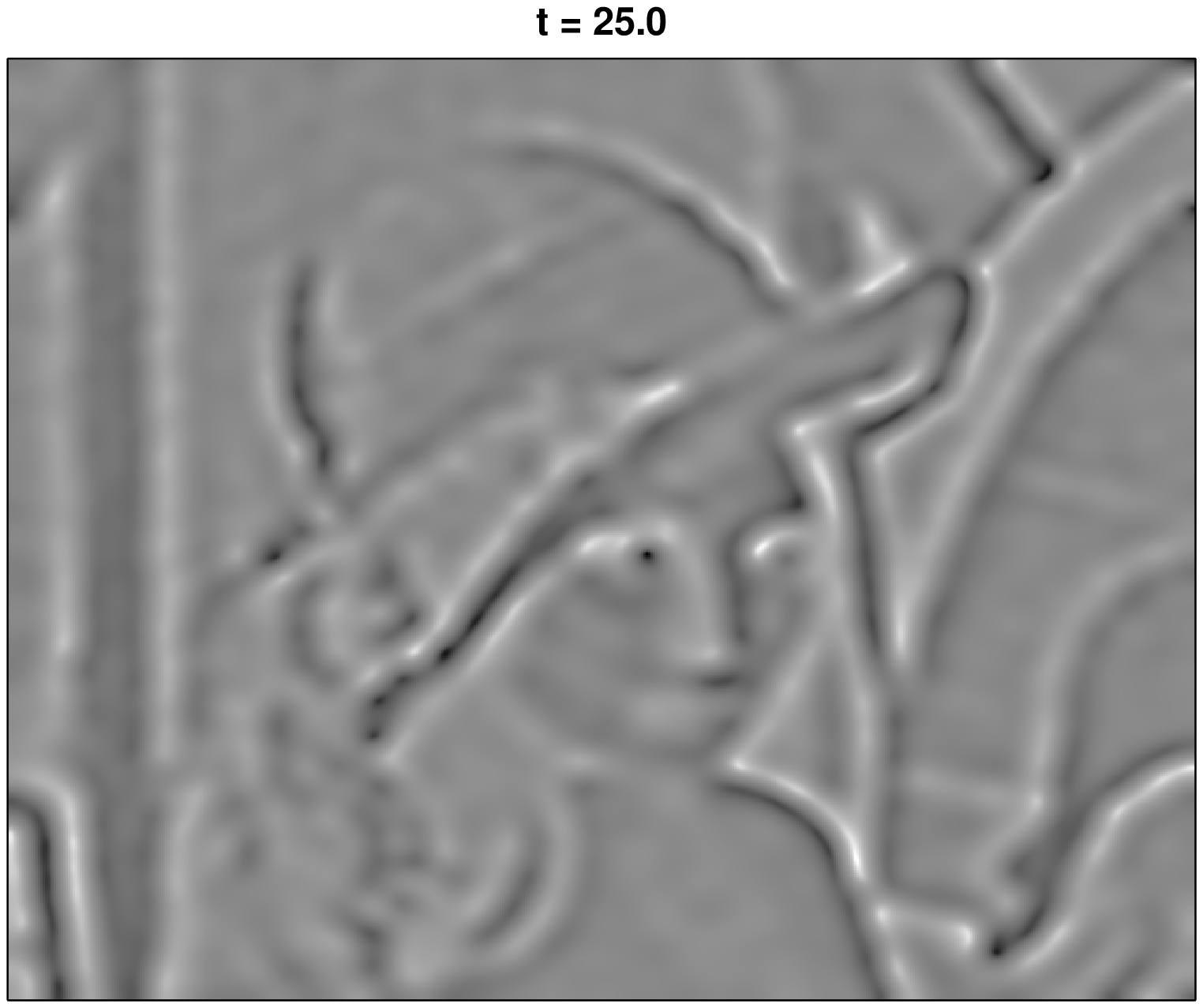}}
\caption{Second component of solution of (\ref{cd35}) at times $t=2.5, 15, 25$ with NCDF5.}
\label{figRII_6b}
\end{figure}

\begin{figure}
\centering
\subfigure[]
{\includegraphics[width=8.5cm,height=7cm]{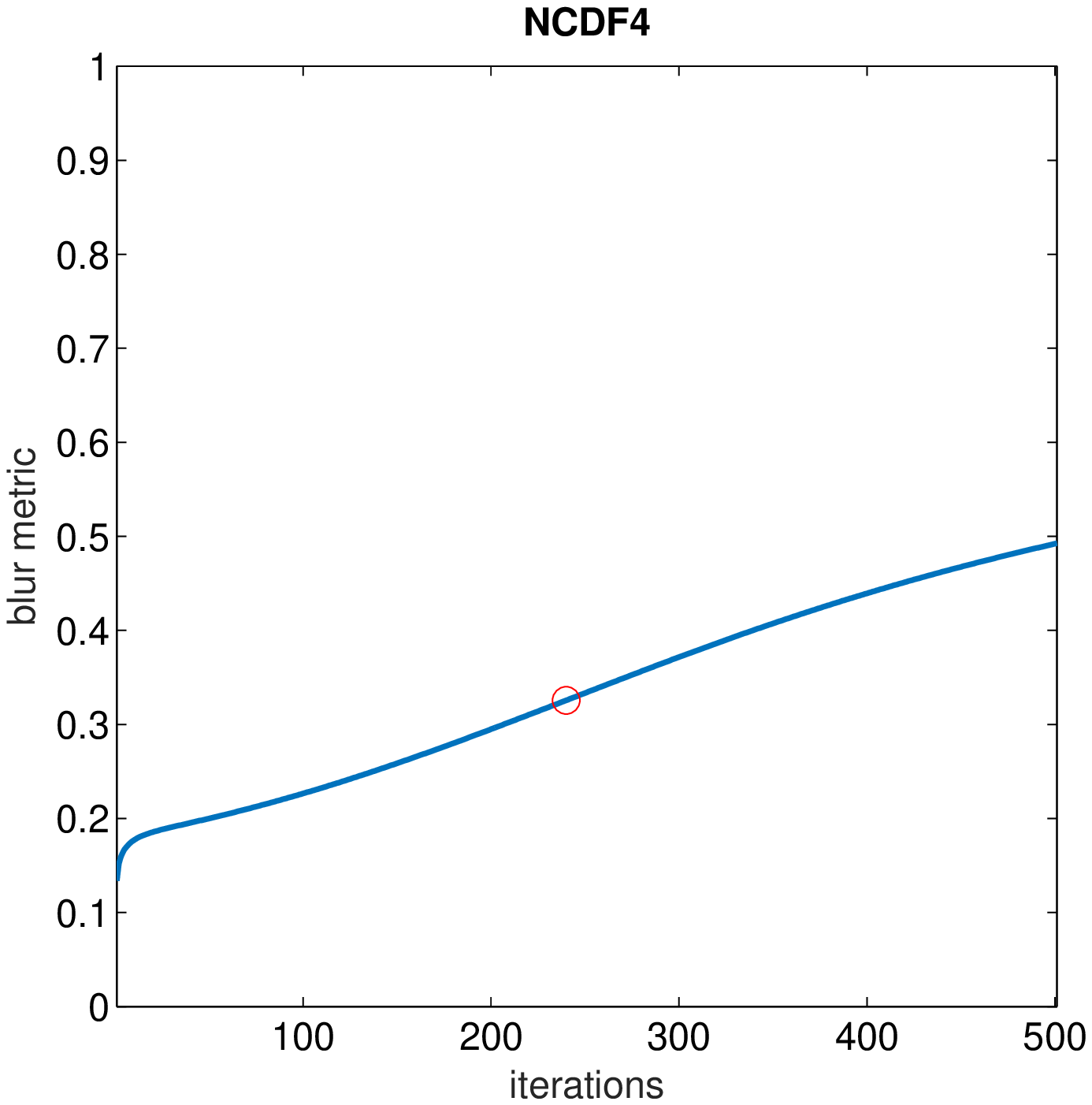}}
\centering
\subfigure[]
{\includegraphics[width=8.5cm,height=8.5cm]{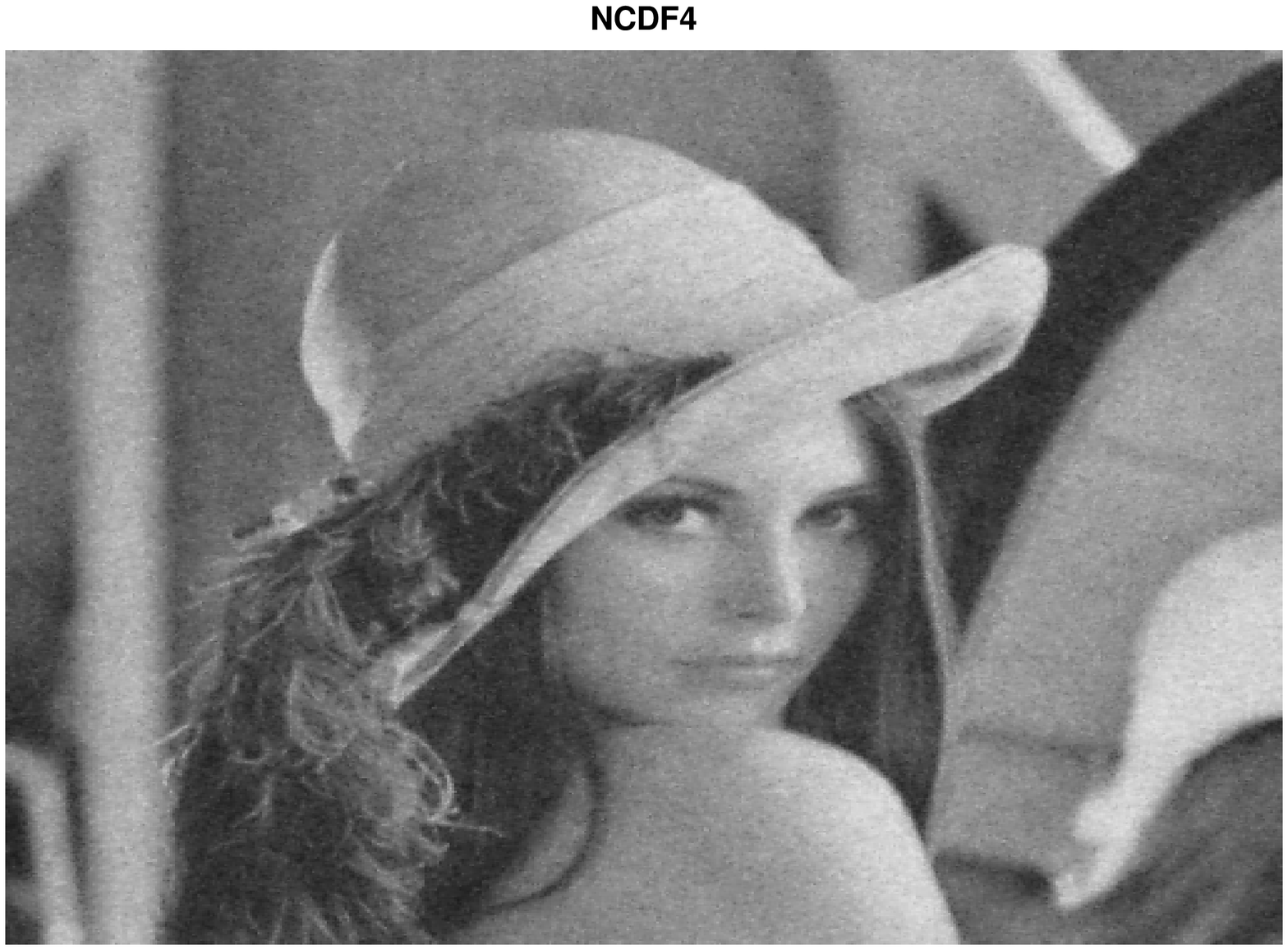}}
\caption{NCDF4: (a) NPB values vs time and (b) first component of the solution of (\ref{cd35}) at the time marked by the small circle.}
\label{figRII_7}
\end{figure}

\begin{figure}[htbp]
\centering
\subfigure[]
{\includegraphics[width=8.5cm,height=7cm]{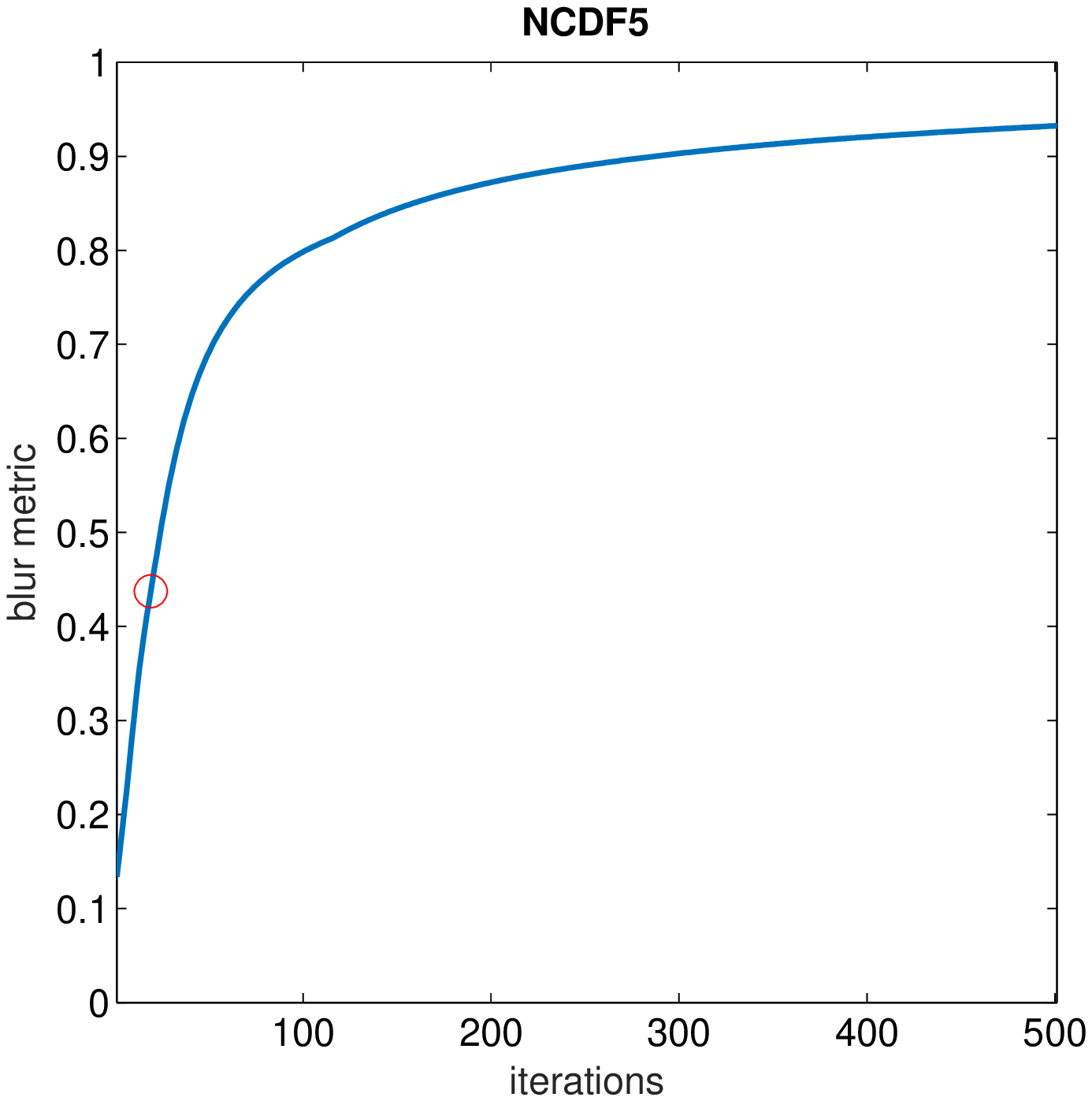}}
\subfigure[]
{\includegraphics[width=8.5cm,height=8.5cm]{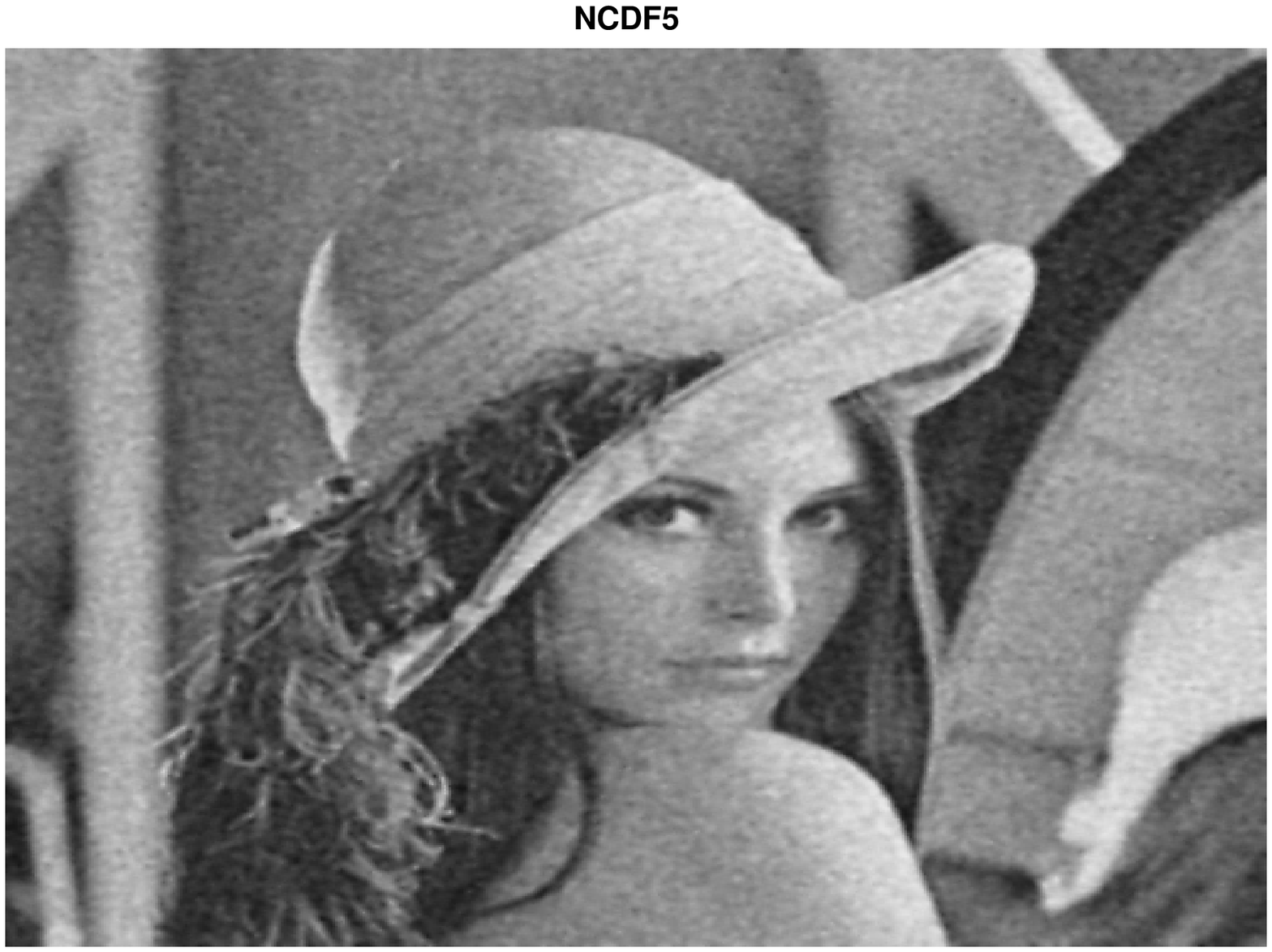}}
\caption{NCDF5: (a) NPB values vs time and (b) first component of the solution of (\ref{cd35}) at the time marked by the small circle.}
\label{figRII_8}
\end{figure}

\begin{figure}[htbp]
\centering
\subfigure[]
{\includegraphics[width=8.5cm,height=7cm]{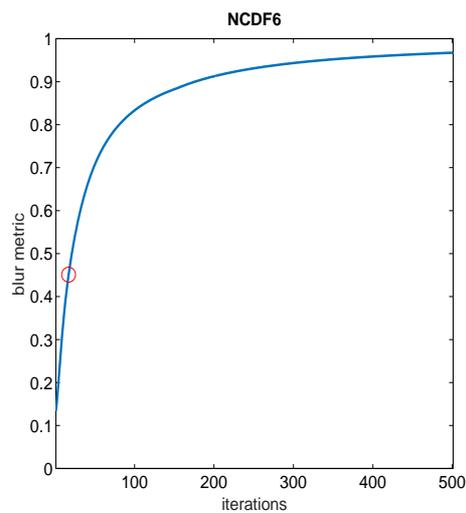}}
\subfigure[]
{\includegraphics[width=8.5cm,height=8.5cm]{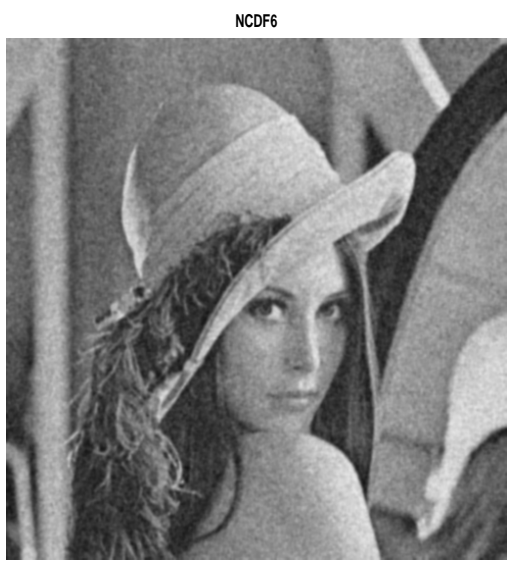}}
\caption{NCDF6: (a) NPB values vs time and (b) first component of the solution of (\ref{cd35}) at the time marked by the small circle.}
\label{figRII_9}
\end{figure}

\section{Concluding remarks}
\label{sec:sec4}
In the present paper nonlinear cross-diffusion systems as mathematical models for image filtering are studied. This is a continuation of the companion paper, Ara\'ujo et al. \cite{ABCD2016}, devoted to the linear case. Here the nonlinearity is introduced through $2\times 2$, uniformly positive definite cross-diffusion coefficient matrices with bounded, globally Lipschitz entries. In the first part of the paper well-posedness of the corresponding IBVP with Neumann boundary conditions is proved, as well as several scale-space properties and the limiting behaviour to the constant average grey value of the image at infinity. The second part is devoted to some numerical comparisons on the performance of the filtering process from some noisy images using three  models distinguished by different choices of the cross-diffusion matrix.
 The computational part 
does not intent to be exhaustive and instead aims to suggest and anticipate some preliminary conclusions that may motivate further research. As in the linear case, the systems incorporate some degrees of freedom. This diversity is mainly represented by the choice of the cross-diffusion matrix. The numerical study performed here makes use of cross-diffusion matrices whose derivation was based on the choices made in Gilboa et al.  \cite{GilboaSZ2004} for the complex diffusion case, combined with the results on linear cross-diffusion. The numerical results reveal that the structure of the diffusion coefficients affects the evolution of the filtering process and the quality in the detection of the edges through one of the components of the system.

Additional lines of future research concern the extension of the cross formulation to study edge-enhancing problems as well as the introduction and analysis of discrete cross-diffusion systems,  as discrete models for image filtering and as schemes of approximation to the continuous problem.

\begin{acknowledgements}
This work was supported by  Spanish Ministerio de Econom\'{\i}a y Competitividad under the Research Grant MTM2014-54710-P.
A. Ara\'ujo and S. Barbeiro were also supported by the Centre for Mathematics of the University of Coimbra -- UID/MAT/00324/2013, funded by the Portuguese Government through FCT/MCTES and co-funded by the European Regional Development Fund through the Partnership Agreement PT2020.
\end{acknowledgements}

\end{document}